\documentclass[11pt]{amsart}
\usepackage[latin1]{inputenc}
\usepackage{amsmath, amssymb, amsfonts, mathtools, amscd, amsthm, color, latexsym, graphicx, tensor, array, newclude, stmaryrd}
\usepackage[alphabetic, nobysame]{amsrefs}
\usepackage[textwidth=20mm]{todonotes}
\usepackage[top=3cm, bottom=3cm, left=3cm, right=3cm]{geometry}
\usepackage{tikz-cd}
\usepackage[shortlabels]{enumitem}

\usepackage{hyperref}      
\hypersetup{pdfpagelabels,
                 plainpages=false,
                    colorlinks=true,       
                    linkcolor=red,
                    citecolor=blue}

\usepackage[nameinlink]{cleveref}

\usepackage{setspace} 

\theoremstyle{plain}
\newtheorem{thm}{Theorem}[section]

\newtheorem{prop}[thm]{Proposition}

\newtheorem{lem}[thm]{Lemma}

\newtheorem{cor}[thm]{Corollary}

\theoremstyle{definition}

\newtheorem{defi}[thm]{Definition}
\newtheorem{rem}[thm]{Remark}

\newcommand{\Qbar}{{\overline{\bbQ}}}

\newcommand{\Spec}{{\rm Spec}}

\newcommand{\Gr}{{\textnormal{Gr}}}

\newcommand{\End}{\mathrm{End}\,}
\newcommand{\Hom}{{\rm Hom}}

\newcommand{\Gal}{{\rm Gal}}

\newcommand{\GL}{{\rm GL}}

\newcommand{\Sym}{\textnormal{Sym}}

\newcommand{\Aut}{\textnormal{Aut}}
\newcommand{\an}{\textnormal{an}}

\newcommand{\id}{\mathrm{id}}

\newcommand{\Def}{{\rm Def}}

\newcommand{\Z}{{\mathbb Z}}
\newcommand{\Q}{{\mathbb Q}}
\newcommand{\R}{{\mathbb R}}
\newcommand{\C}{{\mathbb C}}

\newcommand{\im}{\mathrm{Im}}
\newcommand{\reg}{\mathrm{reg}}
\newcommand{\IH}{\mathit{IH}}

\newcommand{\Simp}{\mathrm{Simp}}
\newcommand{\ob}{\mathrm{ob}}
\newcommand{\et}{\mathrm{\acute{e}t}}
\newcommand{\iso}{\mathrm{Iso}}
\newcommand{\B}{\mathrm{B}}
\newcommand{\dR}{\mathrm{dR}}
\newcommand{\Dol}{\mathrm{Dol}}

\newcommand{\bD}{{\mathbf D}}

\newcommand{\bM}{{\mathbf M}}

\newcommand{\bR}{{\mathbf R}}
\newcommand{\bS}{{\mathbf S}}

\newcommand{\bU}{{\mathbf U}}

\newcommand{\fm}{\mathfrak{m}}

\newcommand{\cA}{{\mathcal A}}
\newcommand{\cB}{{\mathcal B}}
\newcommand{\cC}{{\mathcal C}}
\newcommand{\cD}{{\mathcal D}}
\newcommand{\cE}{{\mathcal E}}
\newcommand{\cF}{{\mathcal F}}
\newcommand{\cG}{{\mathcal G}}
\newcommand{\cH}{{\mathcal H}}

\newcommand{\cL}{{\mathcal L}}
\newcommand{\cM}{{\mathcal M}}

\newcommand{\cO}{{\mathcal O}}

\newcommand{\cQ}{{\mathcal Q}}
\newcommand{\cR}{{\mathcal R}}

\newcommand{\cV}{{\mathcal V}}
\newcommand{\cW}{{\mathcal W}}

\newcommand{\bbP}{{\mathbb P}}
\newcommand{\bbQ}{{\mathbb Q}}

\begin{document}
\title{Cohomology jump loci and absolute sets for singular varieties}
\author{Leonardo A. Lerer}
\address{Weizmann Institute of Science, Rehovot 7610001, Israel} 
\email{a.leonardo.lerer@gmail.com}
\begin{abstract}
We extend the notion of absolute subsets of Betti moduli spaces of smooth algebraic varieties to the case of normal varieties. As a consequence we prove that twisted cohomology jump loci in rank one over a normal variety are a finite union of translated subtori. We show that the same holds for jump loci twisted by a unitary local system in the case where the underlying variety $X$ is projective with $H^1(X,\Q)$ pure of weight one. Lastly, we study the interaction of these loci with Hodge theoretic data naturally associated to the representation variety of fundamental groups of smooth projective varieties.
\end{abstract}
\maketitle


\section{Introduction}

Let $X$ be a complex algebraic variety and $n\in\Z_{\ge 1}$ . The set of semisimple local systems of rank $n$ on $X$ is parametrized by the $\C$-points of a scheme $\bM_{\B}(X,n)$, called the Betti moduli space. The topological properties of the space $X$ are reflected in the structure of $\bM_{\B}(X,n)$ and of naturally defined subsets therein. 
An example of the latter are the {\em cohomology jump loci}, namely the sets
$$\prescript{}{n}\Sigma^i_k(\cF) := \{ L\in \bM_{\B}(X,n)(\C) : \dim_\C H^i(X,L\otimes \cF)\ge k\},$$
for $i,k\in\Z_{\ge 0}$ and $\cF$ is an object of $D^b_c(X,\C)$, the bounded derived category of constructible sheaves on $X$. When $\cF=\C_X$, the constant sheaf of rank one, we simply write $\prescript{}{n}\Sigma^i_k:=\prescript{}{n}\Sigma^i_k(\C_X)$ and, when $n=1$, we write $\Sigma^i_k(\cF):=\prescript{}{1}\Sigma^i_k(\cF)$. These subsets are closed algebraic subvarieties of $\bM_{\B}(X,n)$ and they have been the subject of numerous investigations, starting with Green-Lazarsfeld, Beauville and Catanese (see  \cite{GL87}, \cite{B88}, \cite{C91}).

When $n=1$ the Betti moduli space is a finite union of complex algebraic tori and if $X$ is smooth the cohomology jump loci have a simple form, namely they are translated subtori of $\bM_{\B}(X,1)$. When $\cF\in D^b_c(X,\Qbar)$ is of geometric origin or, more generally, if it defines an absolute $\Qbar$-point (cf. Definition \ref{defi_abs_constr}) the translation is by torsion points. In this generality, the result was proved in \cite[Theorem 1.4.1]{BW20}, building on the previous work of many authors, see the chronology in \cite[Remark 6.13]{BW17b}. 

The main focus of this paper is the study of the structure of such subvarieties in the case where $X$ is allowed to be singular. This has been the topic of some recent works, in which the torsion-translated-subtori property was extended to some classes of singular varieties. In \cite[Theorem 1.1]{ADH} the property is proved for $\Sigma^1_k$ when $X$ is normal. If $X$ is an algebraic variety whose mixed Hodge structure $H^1(X,\Q)$ satisfies $W_0H^1(X,\Q)= 0$, it is proved in \cite[Theorem 1.1]{BR19} that the irreducible components of $\Sigma^i_k$ passing through the identity are subtori. The most general result is found in \cite[Theorem 1.6]{EK} where the torsion-translated-subtori property is proved for $X$ satisfying the same weight condition on $H^1$ and for $\cF\in D^b_c(X,\Qbar)$ satisfying a certain arithmeticity condition, which includes objects coming from geometry (see the discussion at the end of Section 2.1 for a comparison between the notion of arithmeticity and the notion of absolute $\Qbar$-point). Note that a normal variety satisfies $W_0H^1(X,\Q)=0$, since the pullback $H^1(X,\Q)\to H^1(X^{\reg},\Q)$ is injective, where $X^{\reg}$ denotes the smooth locus (see the related discussion in Section 2.3). We extend the results of \cite{EK} in the case of normal varieties and deal with arbitrary twisting objects $\cF\in D^b_c(X,\C)$. More precisely, we prove:

\begin{thm}[\ref{cor_coh_jump_loci}]
Let $X$ be a complex quasi-projective normal variety and $\cF\in D^b_c(X,\C)$. Then, for all $i,k\ge 0$, $\Sigma^i_k(\cF)$ is a finite union of translated subtori. If $\cF\in D^b_c(X,\Qbar)$ is of geometric origin, the translation is by torsion points. 
\end{thm}

The results of \cite{BR19} and \cite{EK} mentioned above suggest that the most general varieties for which we can expect the torsion-translated-subtori property for twisted cohomology jump loci are those varieties $X$ satisfying $W_0H^1(X,\Q)=0$ (we also refer the reader to the article \cite{Saito} by M. Saito for a study on the condition $W_0H^1=0$). In this direction, we prove that if $X$ is a {\em projective} variety satisfying the above weight condition, i.e. whose $H^1$ is pure of weight one, and $E$ is a unitary local system on $X$, then $\Sigma^i_k(E)$ is a finite union of translated subtori. In particular, we do not require $E$ to possess any arithmeticity property. We remark, however, that we do not recover the fact that the translation is by torsion points when $E$ comes from geometry, as proved in \cite{EK}. Our precise statement is as follows:

\begin{thm}[\ref{thmmulti}]
Let $X$ be a complex projective variety such that $W_0H^1(X,\Q)=0$. Let $E$ be a semisimple local system on $X$ such that for one (equivalently, any) resolution of singularities $p:\widetilde{X}\to X$, $p^*E$ is semisimple. Then, for all $i,k\ge 0$, the varieties $\Sigma^i_k(E)$ are a finite union of translated subtori.
\end{thm}

We then study the interaction of cohomology jump loci with Hodge theoretic data which is naturally attached to the Betti moduli space or representation varieties. In rank one, the Zariski tangent space at the identity of $\bM_{\B}(X,1)$ is canonically identified with $H^1(X,\Q)$ and in \cite{EK} it is proved that the tangent space of irreducible components of cohomology jump loci, translated back to the identity, is a sub-MHS of $H^1(X,\Q)$. A subtorus of $\bM_{\B}(X,1)$ whose tangent space is a sub-MHS is called ``motivic" in {\em loc. cit.}. 
Denote by $\bM_{\B}(X,1)^0$ the connected component of the Betti moduli space containing the constant local system.
We prove the following:

\begin{thm}[\ref{cor_motiv_rank_one}]
Let $X$ be a complex quasi-projective normal variety and $\cF\in D^b_c(X,\C)$. The subtori corresponding to the irreducible components of $\Sigma^i_k(\cF)\cap \bM_{\B}(X,1)^0(\C)$ are motivic.
\end{thm}

To prove Theorem 1.1 and 1.3, we build on the work in \cite{BW20} and extend the notion of absolute subsets of Betti moduli spaces to the case of normal varieties using resolutions of singularities. The idea of studying the topology of normal varieties by resolving singularities was already employed for instance in \cite{ADH}. On the other hand, the proof of Theorem 1.2 follows an ``ad-hoc" strategy involving simplicial resolutions. We also remark that our proofs use transcendental methods while the work of Esnault-Kerz uses $p$-adic geometry and Galois-theoretic arguments.

For local systems of higher rank, the picture is more complicated. Let $X$ be a smooth projective variety, $x_0\in X(\C)$ a point and denote by $\bR(X,x_0,n)$ the representation variety of rank $n$, i.e. the fine moduli space parametrizing rank $n$ representations of $\pi_1(X,x_0)$. If $\rho: \pi_1(X,x_0)\to \GL(n,\C)$ is a representation underlying a polarizable complex variation of Hodge structures ($\C$VHS), Eyssidieux and Simpson in \cite{ES} and later Lef\`evre in \cite{L1} have constructed a functorial $\C$MHS on the completed local ring $\widehat{\cO}_\rho$, where $\cO_\rho$ is the local ring at $\rho$ in $\bR(X,x_0,n)$. Note that such mixed Hodge structure is defined on an infinite dimensional vector space, cf. Definition \ref{hodge}. On $\widehat{\cO}_\rho$ there is also another functorial $\C$MHS, which is easier to describe and which is isomorphic to the $\Gr_W$ of the previous $\C$MHS, where $W$ denotes the weight filtration. It corresponds essentially to a bigrading on $\widehat{\cO}_\rho$ as a $\C$-vector space. We shall refer to it as the {\em split} $\C$MHS on $\widehat{\cO}_\rho$ (see Section \ref{higherrank} for its construction). Define the cohomology jump loci in the representation variety by $\widetilde{\Sigma}^i_k:=\pi^{-1}(\Sigma^i_k)$, where $\pi: \bR(X,x_0,n)\to \bM_{\B}(X,n)$ is the GIT quotient map. We prove that, if $k:= \dim_\C H^i(X,L_\rho)$, then the ideal $I\subset\widehat{\cO}_\rho$ cutting the locus $\widetilde{\Sigma}^i_k$ at $\rho$ is a sub-MHS of the split MHS on $\widehat{\cO}_\rho$ (see Theorem \ref{tangent}). In other words, we prove that $I$ is compatible with the bigrading corresponding to the split MHS introduced above. This can be seen as a higher rank analogue of the motivicity result of \cite{EK}.

In the last section, we complement some results of \cite{B09} concerning jump loci associated to the Hodge theory of $H^\bullet(X,L)$ where $X$ is a smooth algebraic variety and $L$ is a unitary local system of rank one (thus underlying a polarizable $\C$VHS of type $(0,0)$). More precisely, we prove:

\begin{thm}[\ref{hodgejump}]
Let $U$ be a smooth complex algebraic variety and $\bU(U,1)\subset \bM_{\B}(U,1)(\C)$ the union of real subtori parametrizing rank one unitary local systems on $U$. Let $\cA$ be the Boolean algebra generated by torsion-translated real subtori of $\bU(U,1)$.

(1) For all $a,i,k\in\Z_{\ge 0}$, the set
$$ W^{a,i,k}:= \{L \in \bU(U,1):   \dim_\C Gr_W^aH^i(U,L)\ge k\} $$
belongs to $\cA$.

(2) Let $j: U\to X$ be an open embedding into a smooth proper variety $X$ such that $X\setminus U$ is a simple normal crossing divisor. For all $p,q,k\in\Z_{\ge 0}$, the set 
$$V^{p,q,k} := \{L \in \bU(U,1):  \dim_\C Gr^F_p \IH^{p+q}(X,L)\ge k\}$$
belongs to $\cA$. Here $\IH$ denotes intersection cohomology.
\end{thm}

\paragraph{\bf{Notation}} A complex algebraic variety is an integral, separated scheme of finite type over $\Spec(\C)$. Unless stated otherwise, a ``variety" is a complex algebraic variety. If $X$ is a complex algebraic variety, we write $LocSys(X)$ for the category of local systems of finite dimensional complex vector spaces on the associated complex analytic space $X^{\mathrm{an}}$. If $K$ is a field, then $D^b_c(X,K)$ denotes the bounded derived category of constructible sheaves of $K$-vector spaces on $X^{\mathrm{an}}$. If $\cC$ is a category, then $\iso(\cC)$ denotes the ``collection" of isomorphism classes of objects of $\cC$. We will consider isomorphism classes of objects only when $\cC$ is an essentially small category, therefore $\iso(\cC)$ will be an honest set (in our case, $\cC$ will be one of the following, for a complex algebraic variety $X$: $D^b_c(X,\C)$, $LocSys(X)$ or the bounded derived category of regular holonomic $\cD_X$-modules). The abbreviation ``dg" (resp. ``dga", ``dgla") will mean ``differential-graded" (resp. ``differential-graded algebra", ``differential-graded Lie algebra").
\medskip
\paragraph{\bf{Acknowledgements}} I wish to thank Nero Budur for introducing me to the topics of this paper as well as for his suggestions and support. I thank Bruno Klingler for the valuable conversations, for his support and for providing comments on an early draft. I also thank Carlos Simpson and Hélène Esnault for their comments. I benefited from email correspondence with Joana Cirici and Louis-Clément Lefèvre. Finally, I would like to thank the referees for their remarks and corrections. This work is part of my PhD thesis.


\section{Absolute sets for normal varieties}
\subsection{Absolute $K$-constructible sets}\label{sec_defi_abs_constr} 
In this section we recall the formalism of absolute constructible subsets on a smooth complex variety, introduced in \cite{BW20}.

Let $X$ be a smooth complex algebraic variety. We denote by $\cD_X\subset\cH om_\C(\cO_X,\cO_X)$ the sheaf of differential operators on $X$. The Riemann-Hilbert correspondence is an equivalence of categories
$$RH: D^b_{rh}(\cD_X)\xrightarrow{\sim} D^b_c(X,\C)$$
between the bounded derived category of regular holonomic $\cD_X$-modules and the bounded derived category of constructible sheaves on $X^\an$ (see e.g. \cite{HTT}).

Given an automorphism $\sigma\in \Aut(\C)$, let $X^\sigma\to X$ be the base change of $X$ via $\sigma$. 
It is an isomorphism of $\Q$-schemes however the induced map on the analytifications is far from being continuous in general. 
Consider the diagram
\begin{equation}\label{disp_diag_base_change}
\begin{tikzcd}
D^b_{rh} (\cD_X)   \arrow[d,"RH"] \arrow[r,"p_\sigma"] & D^b_{rh}(\cD_{X^\sigma}) \arrow[d,"RH"]\\
D^b_c(X,\C) & D^b_c(X^\sigma,\C).
\end{tikzcd}
\end{equation}
where $p_\sigma$  is the pullback of $\cD_X$-modules under the base change over $\sigma$. 
On the sets of isomorphism classes of objects, the maps induced by the diagram (\ref{disp_diag_base_change}) are bijections.
Notice that there is no natural arrow $D^b_c(X,\C) \to D^b_c(X^\sigma,\C)$ making the diagram commute.

For a subset $T\subset \iso(D^b_c(X,\C))$, define
\begin{equation}\label{eqSss}
T^\sigma := RH\circ p_\sigma\circ RH^{-1} (T)\subset \iso(D^b_c(X^\sigma,\C)). 
\end{equation}
If $n\ge 1$ is an integer and $S\subset \bM_{\B}(X,n)(\C)$ is a subset, then, regarding the elements of $S$ as isomorphism classes of semisimple local systems, $S$ determines a subset of $\iso(D^b_c(X^\sigma,\C))$. 
We define $S^\sigma$ as in (\ref{eqSss}). 
Since Galois conjugation preserves flat connections and semisimplicity, we deduce that $S^\sigma$ corresponds to a subset of $\bM_{\B}(X^\sigma,n)(\C)$. Recall that since $\bM_{\B}(X,n)$ is defined over $\Q$ (even over $\Z$), the notion of $K$-constructible subset of $\bM_{\B}(X,n)(\C)$, for a subfield $K\subset \C$, is well-defined. 

\begin{defi}\label{defi_abs_constr}
Let $K\subset \C$ be a subfield. 
\begin{enumerate}
\item A subset $S\subset  \bM_{\B}(X,n)(\C)$ is absolute $K$-constructible (resp. absolute $K$-closed) if it is $K$-constructible (resp. $K$-closed) and if, for all $\sigma\in\Aut(\C)$, the subset $S^\sigma\subset  \bM_{\B}(X^\sigma,n)(\C)$ defined above is $K$-constructible (resp. $K$-closed).
\item An object $\cF\in D^b_c(X,K)$ is an \emph{absolute $K$-point} if for all $\sigma\in \Aut(\C)$, there exists an element $\cG\in D^b_c(X^\sigma,K)$ such that 
$$p_\sigma \circ RH^{-1}(\cF_\C) \cong RH^{-1}(\cG_\C).$$
\end{enumerate}
\end{defi}
\begin{rem}
\begin{enumerate}
\item The above definition differs slightly from the one given in \cite[Definition 6.3.1]{BW20} but is seen to be equivalent, according to \cite[Proposition 4.6]{BLW21}.
We chose to present it in this way to avoid introducing the language of unispaces which will not be needed in the sequel.
\item In the case of a smooth projective variety, the absolute constructible subsets were defined by Simpson in \cite[Section 6]{S93}.
\end{enumerate}
\end{rem}

It is easy to see that if $\cF$ is of geometric origin in the sense of \cite[6.2.4]{BBD} then $\cF$ is an absolute $\Qbar$-point. Conversely, it is conjectured that if $\cF$ is a semisimple perverse sheaf over a smooth variety then $\cF$ being an absolute $\Qbar$-point should imply that $\cF$ is of geometric origin, see \cite[Conjecture 10.2.2 (6)]{BW20}.

\begin{rem}\label{rem_abs_vs_moduli_abs}
There is a more geometric version of the definition of absolute constructibility, called {\em moduli-absolute constructibility} in \cite{BW20}. 
Let $X$ be a smooth {\em quasi-projective} variety, $\overline{X}$ a smooth compactification with normal crossing divisor $D$. 
Nitsure \cite{N93} and Simpson \cite{S94a} constructed a scheme $\bM_{\dR}(\overline{X}/D,n)$ that is a coarse moduli space parametrizing Jordan equivalence classes of semistable logarithmic connections of rank $n$, with poles at most on $D$. 
The Riemann-Hilbert correspondence induces an analytic morphism
\begin{equation}\label{disp_RH_an}
RH^\an: \bM_{\dR}(\overline{X}/D,n)^\an\to \bM_{\B}(X,n)^\an.
\end{equation}
A subset $S\subset \bM_{\B}(X,n)(\C)$ is called {\em moduli-absolute $K$-constructible} \cite[Definition 7.4.1]{BW20} if it is $K$-constructible and $RH^{\an}\circ p_\sigma\circ (RH^{\an})^{-1}(S)\subset  \bM_{\B}(X^\sigma,n)$ is $K$-constructible, for all $\sigma\in \Aut(\C)$. Here $p_\sigma:  \bM_{\dR}(\overline{X}/D,n)\to \bM_{\dR}(\overline{X}^\sigma/D^\sigma,n)$ is defined by pullback of semistable logarithmic connections. 

As pointed out in \cite[Section 5.6]{BLW21}, it is not known in general whether $RH^\an$ is surjective.
This leads to problems with the previous definition of moduli-absoluteness.
However, in rank one, it is certainly known that $RH^\an$ is surjective and the two definitions of absoluteness coincide \cite[Proposition 7.4.4]{BW20}. This fact will be used below in the proof of Proposition \ref{thm_motiv_rank_one}.
\end{rem}

\paragraph{\textbf{Arithmeticity.}}
Definition \ref{defi_abs_constr}(2) is reminiscent of the arithmeticity condition of Esnault-Kerz \cite[Definition 1.4]{EK}, which applies more generally to non-smooth varieties. We recall it here for convenience. Let $F\subset \C$ be a finitely generated field, $X$ a geometrically connected algebraic variety over $F$ and $p$ a prime number. We denote by $X_{\C,\et}$ the \'etale site of $X_\C$. An object $\cF\in D^b_c(X_{\C,\et},\Qbar_p)$ is {\em arithmetic} if there exists a finitely generated field extension $F\subset F'$ such that for all $\sigma\in \Aut(\C/F')$, the complex $\sigma(\cF)$ is isomorphic (in the derived category) to $\cF$. An object $\cF\in D^b_c(X,\C)$ is {\em arithmetic} it there exists a number field $K$ such that $\cF$ descends to $D^b_c(X,K)$ and such that for infinitely many embeddings $K\hookrightarrow \Qbar_p$, $\cF$ induces an arithmetic object $\cF_{\et}\in D^b_c(X_{\C,\et}, \Qbar_p)$. 

Assume now that $X$ is smooth and consider $\cF\in D^b_c(X,\C)$: we do not know whether the assumption that $\cF$ is arithmetic implies that $\cF$ is an absolute $\Qbar$-point, or vice versa. The two notions are equivalent if $\cF$ is a local system of rank one on $X$. Indeed, it follows from \cite[Theorem 5.2]{EK} that in this case $\cF$ is arithmetic if and only if it is torsion. On the other hand, a local system of rank one defines an absolute $\Qbar$-point if and only if it is torsion, by \cite[Theorem 9.1.2]{BW20}. 

\subsection{Construction of absolute subsets} 
The main result of \cite{BW20} establishes absolute $K$-constructibility for a wide collection of ``functorially defined loci" in $\bM_{\B}(X,n)(\C)$. We summarize here what we need. 

Let $X$ be a smooth algebraic variety, $f: X\to Y$ a morphism of algebraic varieties. Fix a subfield $K\subset\C$. Consider the following list of functors defined on the derived category $D^b_c(X,\C)$:
\begin{equation}\label{disp_list_functors}
(\star)_K :\qquad f^*,Rf_*, f^!, Rf_!, \mathcal{RH}om(\cF_\C,\cdot), \cH^k, \psi_g,\phi_g, j_{!*}, \cdot\otimes\cF_\C,
\end{equation}
where $\cF\in D^b_c(X,K)$ is an absolute $K$-point, $g$ is an algebraic function $X\to \C$ and $j:X\hookrightarrow Z$ is an open embedding. For the definition of these functors see e.g. \cite{Dimca}. The above list depends on the subfield $K$ only through $\cF\in D^b_c(X,K)$.

\begin{defi}\label{defi_smooth_composition}
A function $\phi: \iso(D^b_c(X,\C))\to \Z$ is a \emph{smooth composition of functors in ($\star)_K$} if there exist $n\in \Z_{\ge 1}$ and smooth varieties $X_1,\ldots, X_n$ such that $\phi$ is equal (after taking isomorphism classes of objects) to a composition
\begin{multline}\label{disp}
D^b_c(X,\C)\to D^b_c(X_1,\C)\to\ldots \\ \ldots \to D^b_c(X_n,\C)\to D^b_c(pt,\C)\xrightarrow{\cH^i}LocSys(pt)\xrightarrow{\mathrm{rank}} \Z
\end{multline}
where each each map $D^b_c(X_i,\C)\to D^b_c(X_{i+1},\C)$ belongs to the list $(\star)_K$.
\end{defi}

\begin{thm}(\cite[6.4.3, 7.4.4]{BW20})\label{mainres}
Let $X$ be a smooth variety and $\phi: \iso(LocSys(X))\to \Z$ a map obtained as
$$\iso(LocSys(X))\to \iso(D^b_c(X,\C))\xrightarrow{\psi}\Z,$$
where $\psi$ is a smooth composition of functors in ($\star)_K$. Then for all $k,n\in\Z_{\ge 0}$, the sets $\{L\in \bM_{\B}(X,n)(\C): \phi(L) = k\}$ are absolute $K$-constructible.
\end{thm}

\subsection{Absoluteness on normal varieties}\label{par_setup} Let $X$ be a normal complex algebraic variety (one could weaken this assumption by taking $X$ to be unibranch: in that case, the normalization induces an homeomorphism on the underlying topological spaces. All functors considered in the sequel depend only on the homeomorphism type of $X^\an$). 
Let $p: \widetilde{X}\to X$ be a resolution of singularities which is an isomorphism over a Zariski open subset $U\subset X$.
By normality of $X$ and $\widetilde{X}$, the maps $\pi_1(U)\to \pi_1(X)$ and $\pi_1(U)\to \pi_1(\widetilde{X})$ are surjective (the latter map is obtained by considering $U$ as an open in $\widetilde{X}$). 
This follows from \cite[0.7 (B)]{FL80}.
The diagram
$$ 
\begin{tikzcd}
  \pi_1(U) \arrow[r,twoheadrightarrow] \arrow[dr,twoheadrightarrow] &  \pi_1(\widetilde{X}) \arrow[d, "p_*"]\\
&\pi_1(X)
\end{tikzcd}
$$
shows that $p_*$ is surjective. 
This implies that the map $i: \bM_{\B}(X,n)\to \bM_{\B}(\widetilde{X},n)$ obtained by functoriality is actually a closed embedding. 
This is immediate at the level of representation varieties: if $\widetilde{x}\in\widetilde{X}(\C)$ is a base point and $x=p(\widetilde{x})\in X(\C)$, then the representation variety $\bR_{\B}(X,x,n)$ is cut scheme-theoretically inside $\bR_{\B}(\widetilde{X},\widetilde{x},n)$ by the equations $\rho(\gamma)=\mathrm{id}$, for $\gamma\in \ker(\pi_1(\widetilde{X},\widetilde{x})\to \pi_1(X,x))$. 
Moreover, the image of $\bR_{\B}(X,x,n)$ in $\bR_{\B}(\widetilde{X},\widetilde{x},n)$ is invariant under the action of $\GL_n$.
Finally, a surjection at the level of rings of functions induces a surjection on their ring of invariants (for rings in characteristic $0$ and reductive groups acting on them see the proof of Theorem 1.1 in \cite[Ch. 1]{MFK}).

\begin{defi}\label{defi_abs_constr_normal}
Let $K\subset \C$ be a subfield. 
Let $X$ be a normal variety, $p: \widetilde{X} \to X$ a resolution of singularities and $i: \bM_{\B}(X,n)\hookrightarrow \bM_{\B}(\widetilde{X},n)$ the induced closed embedding. A subset $S\subset \bM_{\B}(X,n)(\C)$ is {\em absolute $K$-constructible} if $i(S)$ is absolute $K$-constructible.
\end{defi}

We show that the above notion of absoluteness does not depend on the choice of resolution and that $\bM_{\B}(X,n)$ itself is absolute $\Q$-constructible.

\begin{lem}\label{lem0}
Let $X$ be a normal algebraic variety, $p_1:Y_1\to X$ and $p_2: Y_2\to X$ two resolutions of singularities. Denote by $i_1$ (resp. $i_2$) the closed embedding of $\bM_{\B}(X,n)$ in $\bM_{\B}(Y_1,n)$ (resp. in $\bM_{\B}(Y_2,n)$). Let $S\subset \bM_{\B}(X,n)(\C)$ be a subset such that $i_1(S)$ is absolute $K$-constructible. Then $i_2(S)$ is absolute $K$-constructible.
\end{lem}

\begin{proof}
Dominate $Y_1$ and $Y_2$ by a common resolution $Y_3$, obtaining a diagram
$$
\begin{tikzcd}
 & \bM_{\B}(Y_3,n) &\\
\bM_{\B}(Y_2,n) \arrow[ur] & & \bM_{\B}(Y_1,n) \arrow[ul]\\
 & \bM_{\B}(X,n) \arrow[ur] \arrow[ul]&
\end{tikzcd}
$$
where the maps are closed embeddings. 
Let $i_3: \bM_{\B}(X,n)\to  \bM_{\B}(Y_3,n)$ be the composition of the closed embeddings.
After conjugation by an element $\sigma\in \Aut(\C)$ the varieties $Y_i^\sigma$, $i=1,2$ are still resolution of singularities of $X^\sigma$ and one obtains the same diagram of Betti moduli spaces for the conjugated varieties. 
By the functoriality of the Riemann-Hilbert correspondence and the hypothesis that $i_1(S)$ is absolute $K$-constructible, we deduce that $i_3(S)$ is an absolute $K$-constructible subset of $\bM_{\B}(Y_3,n)(\C)$.
Again by functoriality and the fact that the maps are closed embeddings, we deduce that $i_2(S)$ is absolute $K$-constructible as well.
\end{proof}

In the following proposition we consider, along with the Betti moduli space, the {\em Dolbeault} moduli space $\bM_{\Dol}(Y,n)$ for a smooth projective variety $Y$. Its complex points parametrize polystable Higgs bundles of rank $n$ with vanishing Chern classes. There is an algebraic action of $\C^\times$ on $\bM_{\Dol}(Y,n)$ given by scalar multiplication of the Higgs field. For the definition of these objects we refer to \cite{S92}.

\begin{prop}\label{prop_embedding_betti_moduli}
Let $X$ be a normal quasi-projective variety and $p:\widetilde{X}\to X$ a resolution of singularities. 
Let $n$ be a positive integer and $S\subset \bM_{\B}(\widetilde{X},n)(\C)$ be the image of $\bM_{\B}(X,n)(\C)$ under the closed embedding. 
We have the following:

(1) $S$ is a closed, absolute $\Q$-constructible subset.

(2) Let $U(1):=\{z\in \C: |z| = 1\}$. 
Assume that $X$ is projective. 
Then $S$ is invariant under the $U(1)$-action deduced from the $U(1)$-action on the Dolbeault moduli space $\bM_{\Dol}(\widetilde{X},n)$ by the Simpson correspondence.
\end{prop}

\begin{rem}
(1) The second point in the above proposition is related to a question in \cite[Remark 7.4.2]{BW20} that asks if an absolute $\Qbar$-constructible set in $\bM_{\B}(Y,n)$ (for $Y$ smooth projective) is invariant under the action of $\C^\times$ deduced from the Simpson correspondence. 
The question is still open.

(2) The $U(1)$-invariance of $i(\bM_{\B}(X,n)(\C))$ implies, by the same arguments of \cite[\S 4]{S92}, that fundamental groups of normal projective varieties cannot be isomorphic to rigid lattices in reductive groups that are not of Hodge type. 
Indeed, with the above notations, the map $\bM_{\B}(X,n)\to \bM_{\B}(\widetilde{X},n)$ is a closed embedding hence an isolated point of $\bM_{\B}(X,n)^\an$ will remain isolated in $i(\bM_{\B}(X,n)^\an)\subset \bM_{\B}(\widetilde{X},n)^\an$. 
Therefore it must be $U(1)$-invariant. 
The fact that fundamental groups of normal varieties should obey the same restrictions coming from the theory of harmonic maps is noticed at the end of \cite[Ch. 1, Section 2.3]{ABCKT}.
\end{rem}

\begin{proof}[Proof of Proposition \ref{prop_embedding_betti_moduli}]
(1) The set $S$ coincides with the set of rank $n$ semisimple local systems on $\widetilde{X}$ that are constant along the fibers of $p$. 
Indeed if $L$ is such a local system then $p_*L$ is a local system on $X$ (by connectedness of the fibers of $p$), semisimple and such that  $L=p^*p_*L$. 
By Lemma \ref{lem0}, we can and will assume that $\widetilde{X}$ is quasi-projective and that the morphism $p$ is projective.
Thus, there is a closed embedding $\widetilde{X}\hookrightarrow \bbP^N\times X$ of $X$-schemes. 
Choose also a closed embedding $X\hookrightarrow Y$ inside a smooth variety $Y$ (note that here we use the hypothesis of quasi-projectivity of $X$). 
Denote by $i: \widetilde{X}\hookrightarrow \bbP^N\times Y$ the closed embedding deduced from the above. 
Fix a point $x\in X(\C)$ and consider the closed embedding $f_x: \bbP^N\times\{x\}\hookrightarrow \bbP^N\times Y$. 
If $L$ is a local system of rank $n$ on $\widetilde{X}$ then $L$ is constant on the fiber $p^{-1}(x)$ if and only if $H^0(p^{-1}(x), L|_{p^{-1}(x)}) \cong \C^n$, i.e. if and only if $\dim_\C H^0(\bbP^N\times \{x\}, f_x^*i_*L) = n$. 
Consider the composition of functors
$$LocSys(\widetilde{X})\xrightarrow{i_*} Sh_c(\bbP^N\times Y)\xrightarrow{f_x^*} Sh_c(\bbP^N\times \{x\})\xrightarrow{\mathrm{rank}\,\circ\, H^0}\Z.$$
Here $Sh_c$ denotes the category of constructible sheaves.
All the varieties involved are smooth.
After passing to the derived categories and isomorphism classes of objects, the composite map is a smooth composition of functors in $(\star)_\Q$, in the sense of Definition \ref{defi_smooth_composition}. 
By Theorem \ref{mainres} we conclude that the set $S_x\subset\bM_{\B}(\widetilde{X},n)(\C)$ of rank $n$ local systems that are constant on $p^{-1}(x)$ is absolute $\Q$-constructible.
Moreover it is closed: calling $\phi$ the above composition of functors, we have that $S_x = \{ L\in \bM_{\B}(\widetilde{X},n)(\C): \phi(L)=n\} = \{ L\in \bM_{\B}(\widetilde{X},n)(\C): \phi(L)\ge n\}$, and the latter set is closed by upper semi-continuity. 
We have $S=\bigcap_{x\in X(\C)} S_x$. By Noetherianity of the Zariski topology the intersection is finite.
It is not difficult to check that a finite intersection of absolute $K$-constructible subsets is again absolute $K$-constructible.

\medskip

(2) First observe that the claim is independent of the choice of resolution (dominate two resolutions by a common one and use that the $\C^\times$-action on the moduli spaces commutes with pullbacks). 
Let $p: \widetilde{X}\to X$ be a resolution of singularities. 
Arguing as in the proof of part (1), there exists a finite number of points $x_1,\ldots, x_l\in X(\C)$ such that $S=\bigcap_{j=1}^l S_{x_j}$, where $S_{x_j}\subset \bM_{\B}(\widetilde{X},n)(\C)$ is the closed subset of local systems on $\widetilde{X}$ which are constant on $p^{-1}(x_j)$. 
Each reduced fiber $p^{-1}(x_j)$ is a connected, possibly reducible subscheme of $\widetilde{X}$. 
If its irreducible components are not smooth, by the embedded resolution of singularities \cite[Main Theorem II]{H64} we can further blowup $\widetilde{X}$ along smooth centers contained in $p^{-1}(x_j)$ to obtain a resolution of singularities $q:\widetilde{X}'\to X$ such that $q^{-1}(x_j)$ has smooth irreducible components. 
By the birational invariance of the fundamental group, one also has $\pi_1(\widetilde{X}')\cong \pi_1(\widetilde{X})$ and the local systems on $\widetilde{X}$ that are constant on $p^{-1}(x_j)$ correspond to the local systems on $\widetilde{X}'$ that are constant on $q^{-1}(x_j)$. 
Proceeding like this for each point $x_j$, we can assume from the beginning that all the fibers $p^{-1}(x_j), j=1,\ldots,l$, have smooth irreducible components.

Let $E$ be a semisimple local system of rank $n$ on $\widetilde{X}$ and denote by $(\cE,\theta)$ the Higgs bundle associated to $E$ by the Simpson correspondence. We also fix a K\"ahler metric on $\widetilde{X}$. Write $\cE^{\infty}$ for the $C^\infty$-bundle associated to $\cE$. Recall that $\cE^{\infty}$ has the structure of harmonic bundle: there is an Hermitian metric $K$ on $\cE^{\infty}$, operators $\bar{\partial}$, $\partial_K$, $\theta$ and $\bar{\theta}_K$ acting on sections of $\cE^{\infty}$ such that 
$$D_K = \bar{\partial}+\partial_K+\theta+\bar{\theta}_K$$
is a connection and $E$ is recovered as the sheaf of flat sections: see for instance the discussion in \cite[\S 1, Constructions]{S92}. If $\lambda\in U(1)$, the harmonic bundle associated to $(\cE,\lambda\theta)$ is given by $\cE^{\infty}$ with the {\em same} metric $K$ and operators $\bar{\partial}$, $\partial_K$, $\lambda\theta$ and $\bar{\lambda}\bar{\theta}_K$. We write $D_{K,\lambda} := \bar{\partial}+\partial_K+\lambda\theta+\bar{\lambda}\bar{\theta}_K$.

Assume that $E$ is constant on a fiber $Z:=p^{-1}(x_j)$ for some $1\le j\le l$. We claim that if $\lambda\in U(1)$ then $\lambda E$ is again constant on $Z$ (here $\lambda E$ denotes the local system associated to the Higgs bundle $(\cE,\lambda\theta)$). Write $Z=\cup_{i=1}^n Z_i$ with $Z_i$ the irreducible components of $Z$, which are smooth projective. Since $E|_{Z_i}$ is constant, the functoriality of Simpson correspondence implies that $(\cE|_{Z_i},\theta)$ is isomorphic as Higgs bundle to the trivial Higgs bundle $(\cO_{Z_i}^n, 0)$. This implies that the restriction of $\theta$ to $Z_i$ is identically zero and $E$ corresponds to the sheaf of holomorphic sections $s$ of $\cE^{\infty}$ with $\partial_Ks= 0$. In particular the sheaf of sections $s$ of $\cE^{\infty}|_{Z_i}$ such that $D_{K,\lambda}s = 0$ is independent of $\lambda$. 

Let $\alpha:[0,1]\to Z$ be a continuous loop. We can assume that $\alpha$ is homotopic to a piecewise $C^\infty$ loop that is a composition of $C^\infty$ paths $\alpha_i$ connecting two points $u_i, v_i$, $i=0,\ldots, N$ in the same irreducible component, with $v_i = u_{i+1}$ and such that each $\alpha_i$ is contained in one irreducible component. The monodromy around $\alpha$ of $\lambda E$ is then computed by parallel transport of $\cE^{\infty}_{u_i}$ to $\cE^{\infty}_{v_i}$ along $\alpha_i$ by the connection $D_{K,\lambda}$, for $i=0,\ldots, N$. But we proved above that the parallel transport isomorphism, for points in the same irreducible component, is independent of $\lambda$. By hypothesis, for $\lambda=1$ the parallel transport $\cE^{\infty}_{v_0}\to \cE^{\infty}_{v_0}$ around $\alpha$ is the identity. We conclude that $\lambda E$ is constant on $Z$. Arguing in this way for each fiber $p^{-1}(x_j), j=1,\ldots,l$, we obtain the claim.
\end{proof}

We now generalize Theorem \ref{mainres} to the setting of normal varieties.

\begin{thm}\label{thm_normal_abs_constr}
Let $X$ be a normal quasi-projective variety, $s:X\hookrightarrow Y$ a closed embedding in a smooth variety and $\phi: \iso(LocSys(X))\to \Z$ a function obtained as
$$\iso(LocSys(X))\xrightarrow{s_*} \iso(D^b_c(Y,\C))\xrightarrow{\psi} \Z,$$
where $\psi$ is a smooth composition of functors in ($\star)_K$. Then for all $k,n\in\Z_{\ge 0}$, the sets $\{L\in \bM_{\B}(X,n)(\C): \phi(L) = k\}$ are absolute $K$-constructible.
\end{thm}
\begin{proof}
The proof relies on a simple observation, already used in the proof of Proposition \ref{prop_embedding_betti_moduli}(1). Fix $p:\widetilde{X}\to X$ a resolution of singularities and consider the pullback $p^*: LocSys(X)\to LocSys(\widetilde{X})$. Then $p_*: LocSys(\widetilde{X})\to Sh_c(X)$ provides a left retract to $p^*$, in the sense that the composition $p_*\circ p^*: LocSys(X)\to LocSys(\widetilde{X})\to Sh_c(X)$ is the natural inclusion $LocSys(X)\to Sh_c(X)$, where $Sh_c(X)$ is the category of constructible sheaves on $X$. Indeed, since $p$ has connected fibers then, for all local systems $L$ on $X$, we have $p_*p^*L = L$. 

Write $s_*: LocSys(X)\to D^b_c(Y,\C)$ as the composition
$$LocSys(X)\xrightarrow{p^*} LocSys(\widetilde{X})\xrightarrow{\alpha} D^b_c(Y,\C)$$
where $\alpha$ is
$$LocSys(\widetilde{X})\hookrightarrow D^b_c(\widetilde X,\C) \xrightarrow{\cH^0(R(s_*\circ p_*))} D^b_c(Y,\C).$$
Then, the set  $\widetilde{S}_{n,k}:=\{ E \in \bM_{\B}(\widetilde{X},n)(\C) : \psi\circ\alpha(E) = k\}$ is absolute $K$-constructible by Theorem \ref{mainres}. The set $S_{n,k} := \{ L \in \bM_{\B}(X,n)(\C) : \phi(L) = k\}$ satisfies
$$i(S_{n,k}) = i(\bM_{\B}(X,n)(\C))\cap  \widetilde{S}_{n,k}$$
where $i:\bM_{\B}(X,n)\hookrightarrow \bM_{\B}(\widetilde{X},n)$ is the closed embedding, since $i$ is given exactly by pullback. Since $i(\bM_{\B}(X,n)(\C))$ is absolute $K$-constructible, $i(S_{n,k})$ is as well.
\end{proof}

\begin{cor}\label{cor_coh_jump_loci}
Let $X$ be a normal quasi-projective variety and $\cF\in D^b_c(X,\C)$. The cohomology jump loci $\Sigma_k^i(\cF)$ in rank one are a finite union of translated subtori. If there exists a closed embedding $s:X\to Y$ in a smooth variety such that $s_*\cF\in D^b_c(Y,\C)$ is an absolute $\Qbar$-point (in particular when $\cF$ is of geometric origin), the translation is by torsion points.
\end{cor}
\begin{proof}
Let $s:X\to Y$ be a closed embedding in a smooth variety $Y$ and apply the previous theorem to the map $rank\circ H^i \circ (\otimes s_*\cF) \circ s_*: \iso(LocSys(X))\to \Z$ to conclude that $\Sigma_k^i(\cF)$ is absolute $\C$-constructible (resp. absolute $\Qbar$-constructible) when $\cF\in D^b_c(X,\C)$ (resp. when $s_*\cF$ is an absolute $\Qbar$-point of $Y$). For a smooth variety, the absolute $\C$-constructible (resp. $\Qbar$-constructible) subsets in rank one belong to the Boolean algebra of translated (resp. torsion-translated) subtori by \cite[Theorem 9.1.2]{BW20}. The same is true when $X$ is normal by the embedding of Betti moduli spaces. Since the cohomology jump loci are closed, the corollary follows.
\end{proof}
\begin{rem}
The above arguments break down if we only assume that $X$ is an algebraic variety with $W_0H^1(X,\Q) = 0$ since a resolution of singularities will have disconnected fibers as soon as $X$ has multibranch singularities. Moreover, the condition on the weight does not imply that $\pi_1(X^{\reg})\to \pi_1(X)$ is surjective (equivalently, that $\pi_1(\widetilde{X})\to \pi_1(X)$ is surjective, where $\widetilde{X}$ is a resolution of singularities or just the normalization). Indeed, an example due to Koll\a'ar and Wang appearing in \cite[Example 4.6]{BR19}, exhibits a variety $X$ with $W_0H^1(X,\Q)=0$ which topologically is a fiber bundle over an elliptic curve $E$ with fiber a nodal cubic $C$. By the exact sequence of a fibration, we have that $\pi_1(X)$ is an extension of $\pi_1(E)$ by $\pi_1(C)\cong \Z$. The normalization $\hat{X}$ of $X$ is a $\bbP^1$-bundle over $E$ hence the projection to $E$ induces an isomorphism $\pi_1(\hat{X})\xrightarrow{\sim} \pi_1(E)$. Therefore $\pi_1(\hat{X})\to \pi_1(X)$ cannot be surjective.

However, by \cite[Section 7]{EK}, $W_0H^1(X,\Q) = 0$ is equivalent to the surjectivity of the map $\pi_1^{\mathrm{ab}}(\hat{X})\to \pi_1^{\mathrm{ab}}(X)$.
\end{rem}


\section{Projective varieties with $W_0H^1=0$}
The aim of this section is to prove the following:
\begin{thm}\label{thmmulti}
Let $X$ be a projective variety such that $W_0H^1(X,\Q)=0$. Let $E$ be a semisimple local system on $X$ such that for one (equivalently, any) resolution of singularities $p:\widetilde{X}\to X$, $p^*E$ is semisimple. Then, for all $i,k\ge 0$, the varieties $\Sigma^i_k(E)$ are a finite union of translated subtori.
\end{thm}
The strategy of the proof follows the same idea of the proof of \cite[Theorem 1.5]{BR19}, which reduces the problem to the construction of an additional grading on a certain dg Lie algebra and on a dg module over it. To achieve this, we will use simplicial techniques, drawing heavily from constructions by R. Hain in \cite[\S 5]{H87}. The key geometrical input is the formality of the $C^\infty$ de Rham complex with coefficients in a semisimple local system over a compact K\"ahler manifold, due to C. Simpson in \cite[\S 2, Cohomology]{S92}.

\subsection{Preliminaries on simplicial de Rham theory}
The main references for this section are R. Hain's article \cite{H87} and S. Halperin's lecture notes \cite{H83}.
\begin{defi}\label{locsimp}
Let $K$ be a simplicial set. A local system $L$ on $K$ with values in a category $\cC$ is the data of objects $L_\sigma\in\ob(\cC)$, for all simplices $\sigma\in K$, and of maps $\partial_i: L_\sigma\to L_{\partial_i\sigma}$, $s_i: L_\sigma\to L_{s_i\sigma}$ satisfying the simplicial relations of face maps and degeneracy maps respectively, see \cite[12.15]{H83}.
\end{defi}

Given a simplicial set $K$ and a local system $L$ of $\C$-vector spaces on $K$, we denote by $A^\bullet(K,L)$ (resp. $A_\infty^\bullet(K,L)$) the dg vector space of complex polynomial (resp. $C^\infty$) differential forms on $K$ with values in $L$ \cite[Definition 13.10]{H83}. If $L$ is a local system of $\C$-algebras then $A^\bullet(K,L)$ is naturally a dga. If $X$ is a topological space we denote by $\Simp_\bullet  X$ the simplicial set of singular simplices $\{\Delta^n\to X, n\ge 0\}$. Given a  local system $L$ of $\C$-vector spaces on $X$, it induces via pullback a local system $L_\bullet$ on $\Simp_\bullet  X$, in the sense of Definition \ref{locsimp}. We denote by $\Omega_{\dR}(X,L):=A^\bullet(\Simp_\bullet X, L_\bullet)$ the dg-vector space of complex polynomial differential forms on $\Simp_\bullet X$. Lastly, if $X_\bullet$ is a simplicial topological space and $L_\bullet$ is a $\C$-local system on $X_\bullet$, we denote by $A^\bullet(X_\bullet, L_\bullet)$ the cosimplicial dg vector space where the $n$-th term is given by $A^\bullet(\Simp_\bullet X_n, L_n)$ and the maps between them are pullbacks of the simplicial maps of $X_\bullet$.

\begin{defi}(\cite[5.2]{H87})
(1) Let $C$ be a cosimplicial $\C$-vector space. Denote by $E^n(\Delta^k)$ the vector space of complex polynomial $n$-forms on the simplex $\Delta^k$. The {\em de Rham complex} of $C$, denoted by $\bD C$, is the dg vector space whose $n$-th piece, for $n\ge 0$, is given by the set of elements $(w_k)\in\prod_{k\ge 0} E^n(\Delta^k)\otimes_\C C_k$ such that, for all nondecreasing maps $f:[k]\to [j]$ we have 
$$(f^*\otimes\id)w_j = (\id\otimes f_*)w_k\in E^n(\Delta^k)\otimes_\C C_j.$$
The differential $d: \bD^nC \to \bD^{n+1}C$ is given by $d\otimes \id$. We also define $\bS C$ to be the complex of vector spaces with same terms as $C$ and differential given by the alternating sum of the cosimplicial morphisms.

(2) Let $C^\bullet$ be a cosimplicial complex of $\C$-vector spaces. Define the {\em de Rham complex} $\bD C^\bullet$ of $C^\bullet$ to be the total complex of the double complex with $(i,j)$-piece given by $\bD^i C^j$, with horizontal differential given by the de Rham differential $\bD^i C^j\to \bD^{i+1} C^j$ and vertical differential $\bD^i C^j\to \bD^{i} C^{j+1}$ induced by that of the complex $C^\bullet$, for all $i,j\ge 0$. We also denote by $\bS C^\bullet$ the simple complex associated to the double complex with $(i,j)$-piece given by $\bS^i C^j$, with horizontal differentials those induced by $\bS$, and vertical differentials equal to the ones of $C^\bullet$.
\end{defi}

\begin{rem}
(1) We gave the above definitions in the context of cosimplicial $\C$-vector spaces because this will be enough for our purposes. In \cite[5.2]{H87}, the definition deals more generally with cosimplicial abelian groups and $\Q$-differential forms.  

(2) The de Rham functor is not exact in general, but it is exact on the subcategory of split cosimplicial abelian groups over simplicial sets, see \cite[Definition 5.3, Lemma 5.3.4]{H87}. All the cosimplicial abelian groups that we shall consider will be of that form and this will be essential in the proofs.
\end{rem}
\begin{lem}\label{simp1}
Let $X_\bullet$ be a split simplicial space with an augmentation $\epsilon:X_\bullet\to X$ which is of cohomological descent. Let $L$ be a local system of $\,\C$-vector spaces (resp. $\C$-algebras) and $L_\bullet$ the pullback simplicial local system on $X_\bullet$. There is a natural dg (resp. dga) quasi-isomorphism $\Omega_{\dR}(X, L)\to\bD A^\bullet(X_\bullet, L_\bullet)$.
\end{lem}
\begin{proof}

Consider the constant cosimplicial dg vector space $A^\bullet(\Simp_\bullet X,L)^{\mathrm{const}}$, where the $n$-th term is $A^\bullet(\Simp_\bullet X,L)$, for all $n\ge 0$. There is a natural pullback of cosimplicial dg vector spaces $A^\bullet(\Simp_\bullet X,L)^{\mathrm{const}}\to A^\bullet(X_\bullet, L_\bullet)$. Consider the following commutative square:
$$
\begin{tikzcd}
\bD A^\bullet(\Simp_\bullet X, L)^{\mathrm{const}}\arrow[r]\arrow[d] & \bD A^\bullet(X_\bullet, L_\bullet) \arrow[d]\\
\bS A^\bullet(\Simp_\bullet X, L)^{\mathrm{const}}\arrow[r] & \bS A^\bullet(X_\bullet, L_\bullet)
\end{tikzcd}
$$
The cosimplicial dg vector space $A^\bullet(\Simp_\bullet X, L)^{\mathrm{const}}$ (resp. $A^\bullet(X_\bullet, L_\bullet)$) is split over the constant simplicial set (resp. over the simplicial set $\pi_0(X_\bullet)$). Therefore \cite[Lemma 5.3.8]{H87} implies that the vertical maps are quasi-isomorphisms. The two dg vector spaces at the bottom are representatives of $R\Gamma(L)$ and $R\Gamma(X_\bullet, L_\bullet)$ respectively, and the map between them is induced by the adjunction $L\to \epsilon_*\epsilon^*L$. By the cohomological descent hypothesis, the bottom arrow is then a quasi-isomorphism. We conclude the proof by remarking that the top left object is quasi-isomorphic to $A^\bullet(\Simp_\bullet X, L)=\Omega_{\dR}(X, L)$ by integration of differential forms.
\end{proof}
\begin{lem}\label{simp2}
Let $X_\bullet$ be a split simplicial manifold and $L_\bullet$ a local system of $\C$-vector spaces (resp. $\C$-algebras) on $X_\bullet$. Let $\cE^\bullet(X_k,L_k)$ be the usual de Rham complex of smooth forms with values in $L_k$, and $\cE^\bullet(X_\bullet,L_\bullet)$ the resulting cosimplicial dg vector space (resp. dga). Then there is a natural dg (resp. dga) quasi-isomorphism between $\bD A^\bullet(X_\bullet, L_\bullet)$ and $\bD \cE^\bullet(X_\bullet, L_\bullet)$.
\end{lem}
\begin{proof}
The argument of \cite[15.19]{H83} applies to twisted coefficients and shows that if $M$ is a manifold and $L$ is a local system on $M$ then there are functorial quasi-isomorphisms of dg vector spaces $A^\bullet(\Simp_\bullet M,L)\to A_\infty^\bullet(\Simp_\bullet M,L)\leftarrow\cE^\bullet(M,L)$. This implies that there are quasi-isomorphisms of cosimplicial dg vector spaces $A^\bullet(X_\bullet,L_\bullet)\to A_\infty^\bullet(X_\bullet,L_\bullet)\leftarrow\cE^\bullet(X_\bullet,L_\bullet)$. Moreover, since $X_\bullet$ is a split simplicial space, it follows by functoriality of $A^\bullet(-),A_\infty(-),\cE^\bullet(-)$ that all three terms are split over the simplicial set $\pi_0(X_\bullet)$. We then apply the functor $\bD$ and use \cite[5.3.7]{H87} to conclude. When $L_\bullet$ is a simplicial local system of $\C$-algebras, the quasi-isomorphisms are of dga's.
\end{proof}

The proof of Theorem \ref{thmmulti} will be devoted to the construction of a certain dga {\em pair}, namely a couple $(A,M)$ where $A$ is a dga and $M$ is a dg $A$-module. For the basic formalism of dgla-pairs see \cite[\S 3]{BW15}, the formalism of dga-pairs being essentially the same. If $A$ (resp. $A'$) is a cosimplicial dga and $M$ (resp. $M'$) is a cosimplicial dg module over $A$ (resp. $A'$), a morphism of cosimplicial pairs $\phi: (A,M)\to (A',M')$ is the datum of a morphism $\phi_1: A\to A'$ of cosimplicial dga's and a morphism $\phi_2: M\to M'$ of cosimplicial dg vector spaces which is also a morphism of cosimplicial $A$-modules, where $M'$ has the structure of $A$-module induced by $A\to A'$. We say that $\phi$ is a quasi-isomorphism if both $\phi_1$ and $\phi_2$ are quasi-isomorphisms.

\begin{lem}\label{pairs0}
Let $(A,M)$ be a dga pair. Then $A\oplus M$ has a natural structure of dga, functorial in the pair $(A,M)$. Moreover, a morphism $(A,M)\to (A',M')$ is a quasi-isomorphism if and only if $A\oplus M\to A'\oplus M'$ is quasi-isomorphism of dga's.
\end{lem}
\begin{proof}
The construction is well-known and appears in the context of Lie algebras and $L_\infty$-algebras in \cite{L04} and \cite[Theorem 2.5]{BR19}. The differential on $A\oplus M$ is given by $d(a,m)=(da,dm)$ and multiplication is given by the rule $(a_1,m_1)(a_2,m_2)=(a_1a_2, a_1m_2+a_2m_1)$. The proofs in {\em loc. cit.} translate easily in this situation.
\end{proof}

\begin{lem}\label{pairs}
Let $A$ (resp. $A'$) be a cosimplicial dga, $M$ (resp. $M'$) a cosimplicial dg-module over $A$ (resp. $A'$). Assume that $A,A',M,M'$ are split cosimplicial dg-vector space over a simplicial set. Let $\phi: (A,M)\to (A',M')$ be a quasi-isomorphism of cosimplicial dga-pairs. Then the dga-pair $(\bD A,\bD M)$ is quasi-isomorphic to $(\bD A', \bD M')$.
\end{lem}
\begin{proof}
We use Lemma \ref{pairs0} to endow $A\oplus M$ and $A'\oplus M'$ with the structure of cosimplicial dga's. Then $\phi$ induces a quasi-isomorphism of cosimplicial dga's, and by \cite[5.3.7]{H87} we have a quasi-isomorphism $\bD(A\oplus M)\to \bD(A'\oplus M')$. One checks that the dga structure given by Lemma \ref{pairs0} on $\bD A\oplus \bD M$ (resp. $\bD A'\oplus \bD M'$) makes it isomorphic to $\bD(A\oplus M)$ (resp. $\bD(A'\oplus M')$). Again by {\em loc. cit.} one deduces that the pair $(\bD A,\bD M)$ is quasi-isomorphic to $(\bD A', \bD M')$.
\end{proof}
\medskip
\subsection{Proof of Theorem \ref{thmmulti}}
$$$$
Let $(A,M)$ be a dga (or dgla) pair. A {\em grading} $W_\bullet$ of $(A,M)$ is a grading of vector spaces $A^n =\oplus_{j\ge 0} W_jA^n, M^n =\oplus_{j\ge 0} W_jM^n$, for all $n\ge 0$, and such that:
\begin{enumerate}
\item $d(W_iA^n)\subset W_iA^{n+1}$ and $W_iA^n\cdot W_jA^m \subset W_{i+j}A^{n+m}$, for all $i,j\ge 0$, $n,m\ge 0$,
\item $d(W_iM^n)\subset W_iM^{n+1}$ and $W_iA^n\cdot W_jM^m \subset W_{i+j}M^{n+m}$, for all $i,j\ge 0$, $n,m\ge 0$.
\end{enumerate}
A grading $W_\bullet$ on $(A,M)$ induces a grading on the cohomology $H^\bullet A$. The following result is essentially proved in \cite{BR19}. Since it is not stated like this in {\em loc. cit.}, we will give a sketch of the proof at the end of the section.
\begin{thm}[\cite{BR19}]\label{brinfty}
Let $X$ be a complex algebraic variety with $W_0H^1(X,\Q)=0$ and let $E$ be a local system on $X$. Fix integers $i,k\ge 0$ and a point $x\in\Sigma^i_k(E)$. We denote by $L_x$ the rank one local system corresponding to $x$. Assume that the dgla pair $(\Omega_{\dR}(\End L_x),\Omega_{\dR}(E\otimes L_x))$ is quasi-isomorphic to a dgla pair $(A,M)$ with a grading $W_\bullet$ such that $W_\bullet$ induces a splitting of the complexified Deligne weight filtration on $H^\bullet A = H^\bullet(X,\C)$. Then the irreducible components of $\Sigma^i_k(E)$ passing through $x$ are translated subtori.
\end{thm}
Let $L$ be a local system of rank one such that $L\in \Sigma^i_k(E)$. The rest of the proof will be devoted to the construction of a pair $(A,M)$ quasi-isomorphic to $(\Omega_{\dR}(\End L),\Omega_{\dR}(E\otimes L))$ and endowed with a grading $W_\bullet$ as in Theorem \ref{brinfty}. 

Let $X_\bullet$ be a split simplicial resolution of $X$, where each $X_k$ is a smooth projective scheme and $X_0$ a resolution of singularities of $X$. We denote by $\epsilon: X_\bullet\to X$ the augmentation. Let $E_n := \epsilon_n^*E$ be the pullback of $E$ to $X_n$. By hypothesis $E_0$ is semisimple. Moreover, $E_n$ is the pullback of $E_0$ along any of the simplicial maps between $X_n$ and $X_0$. Simpson's theory implies that the pullback of a semisimple local system along a map of compact K\"ahler manifolds remains semisimple: indeed a harmonic metric on the local system pulls back to a harmonic metric. We deduce that $E_n$ is semisimple, for all $n\ge 0$.

Write $L_n:=\epsilon_n^*L$, $n\ge 0$. By Lemma \ref{simp1} and \ref{simp2}, the dga pair $(\Omega_{\dR}(\End L),\Omega_{\dR}(E\otimes L))$ is quasi-isomorphic to the dga pair $(\bD \cE^\bullet (\End L_\bullet), \bD \cE^\bullet (E_\bullet\otimes L_\bullet))$. We will construct $(A,M)$ and the grading $W_\bullet$ from the latter. The idea is an adaptation to the simplicial context of the strategy employed in \cite[\S 6, Main theorem, 1st proof]{DGMS} and \cite[Lemma 2.2]{S92}. We also remark that we will actually construct a dga-pair $(A,M)$ quasi-isomorphic to the dga-pair $(\Omega_{\dR}(\End L),\Omega_{\dR}(E\otimes L))$: the dgla structure will be induced on both sides by the commutator of the dga structure.

At each level $X_k$, we fix a K\"ahler metric on $X_k$ and an harmonic metric on $L_k$ and $E_k$. This induces an harmonic metric on $(\End L)_k$ and on $L_k\otimes E_k$, by the procedure described in \cite[\S 1 Functoriality]{S92}. Let $D, D', D''$ the differential operators on the $C^\infty$ sections of $L_k\otimes E_k$ induced by the harmonic metric \cite[\S 1 Constructions]{S92} and also their extension to the $C^\infty$ differential forms $\cE^\bullet(L_k\otimes E_k)$, induced both by the harmonic metric and the K\"ahler metric. Let $f:X_{k+1}\to X_k$ be any of the simplicial maps. By functoriality of Simpson's correspondence between flat bundles, harmonic bundles and Higgs bundles we have that, for any $j\ge 0$, the following square commutes:
$$
\begin{tikzcd}
\cE^{j+1} (E_k\otimes L_k) \arrow[r, "f^*"] & \cE^{j+1}(E_{k+1}\otimes L_{k+1}) \\
\cE^j (E_k\otimes L_k) \arrow[r, "f^*"]\arrow[u, "\star "] & \cE^j (E_{k+1}\otimes L_{k+1})\arrow[u, "\star "]
\end{tikzcd}
$$

where $\star$ is any of the operators $D,D',D''$. In particular the pullback $f^*$ sends $\ker(D')^j_k:=\ker(D' : \cE_k^j\to\cE_k^{j+1})$ to $\ker(D')^j_{k+1}:=\ker(D' : \cE_{k+1}^j\to\cE_{k+1}^{j+1})$ and the square
$$
\begin{tikzcd}
\ker(D')^{j+1}_k \arrow[r, "f^*"] & \ker(D')^{j+1}_{k+1} \\
\ker(D')^j_k \arrow[r, "f^*"]\arrow[u, "D'' "] & \ker(D')^j_{k+1}\arrow[u, "D'' "]
\end{tikzcd}
$$
commutes. We deduce that there is a morphism, given by the natural inclusion, of cosimplicial dg-vector spaces:
\begin{small}
$$
\begin{tikzcd}
\vdots                                       & \vdots                                                                         &        \\
\ker(D')^{1}_0 \arrow[r, shift left]\arrow[r, shift right] \arrow[u, "D'' "] &\ker(D')^{1}_1 \arrow[r]\arrow[r, shift left=2]\arrow[r, shift right=2] \arrow[u, "D'' "] & \ldots \\
\ker(D')^{0}_0 \arrow[r, shift left]\arrow[r, shift right] \arrow[u, "D'' "] & \ker(D')^{0}_1 \arrow[r]\arrow[r, shift left=2]\arrow[r, shift right=2] \arrow[u, "D'' "]  & \ldots \\      
\end{tikzcd}
\qquad\hookrightarrow\qquad
\begin{tikzcd}
\vdots                                       & \vdots                                                                         &        \\
\cE^1(E_0\otimes L_0) \arrow[r, shift left]\arrow[r, shift right] \arrow[u, "D"] & \cE^1(E_1\otimes L_1) \arrow[r]\arrow[r, shift left=2]\arrow[r, shift right=2] \arrow[u, "D"] & \ldots \\
\cE^0(E_0\otimes L_0) \arrow[r, shift left]\arrow[r, shift right] \arrow[u, "D"] & \cE^0(E_1\otimes L_1) \arrow[r]\arrow[r, shift left=2]\arrow[r, shift right=2] \arrow[u, "D"]  & \ldots \\      
\end{tikzcd}
$$
\end{small}
By \cite[Lemma 2.2]{S92} the induced morphism on the columns is a quasi-isomorphism. We perform the same construction with $\End L_\bullet$ instead of $E_\bullet\otimes L_\bullet$, and remark that we obtain a quasi-isomorphism of \emph{pairs}. By Lemma \ref{pairs}, the pairs $(\bD \cE^\bullet (\End L_\bullet), \bD \cE^\bullet (E_\bullet\otimes L_\bullet))$ and $(\bD \ker_{D'}(\End L_\bullet), \bD \ker_{D'}(E_\bullet\otimes L_\bullet))$ are quasi-isomorphic (the splitness hypothesis of {\em loc. cit.} is satisfied since all the cosimplicial vector spaces are split over $\pi_0(X_\bullet)$). We now employ the same strategy to go from the cosimplicial dga pair $(\ker_{D'}(\End L_\bullet), \ker_{D'}(E_\bullet\otimes L_\bullet))$ to its vertical $D''$-cohomology, using again the commutation of pullbacks with the operators $D,D',D''$ and \cite[Lemma 2.2]{S92}. Therefore we obtain a quasi-isomorphism of cosimplicial dga pairs between $(\ker_{D'}(\End L_\bullet), \ker_{D'}(E_\bullet\otimes L_\bullet))$ and $(H^\bullet(X_\bullet, \End L_\bullet), H^\bullet(X_\bullet, E_\bullet\otimes L_\bullet))$. We now remark that on this last pair there is a clear candidate grading $W$: namely, we define $W_kH^\bullet(X_\bullet, \End L_\bullet)$ to be the $k$-th row
$$
W_kH^\bullet(X_\bullet, \End L_\bullet):=
\begin{tikzcd}
{\big{[}} H^k(\End L_0) \arrow[r, shift left]\arrow[r, shift right] & H^k(\End L_1) \arrow[r]\arrow[r, shift left=2]\arrow[r, shift right=2] & \ldots {\big{]}}
\end{tikzcd}
$$
and similarly for $H^\bullet(X_\bullet, E_\bullet\otimes L_\bullet)$, for all $k\ge 0$.

The sought-for pair $(A,M)$ is then given by $(\bD H^\bullet(X_\bullet, \End L_\bullet), \bD H^\bullet(X_\bullet, E_\bullet\otimes L_\bullet))$ and the grading $W$ by
\begin{align*}
W_k \bD H^\bullet(X_\bullet, \End L_\bullet) &:= \bD W_kH^\bullet(X_\bullet, \End L_\bullet),\\
W_k \bD H^\bullet(X_\bullet, E_\bullet\otimes L_\bullet) &:= \bD W_kH^\bullet(X_\bullet, E_\bullet\otimes L_\bullet).
\end{align*}
It is a straightforward computation to check that $d(W_iA^n)\subset W_iA^{n+1}$ and $W_iA^n\cdot W_jA^m \subset W_{i+j}A^{n+m}$. On $H^\bullet A$, the filtration associated to the grading is the weight filtration: indeed, forgetting the vertical zero differentials, the double complex associated to the cosimplicial dga $H^\bullet(X_\bullet, \End L_\bullet) = H^\bullet(X_\bullet, \C_\bullet)$  is the first page of the spectral sequence associated to the complexified Deligne weight filtration on $H^\bullet(X,\C)$, by \cite[8.1.19.1]{D73}.\hfill$\qed$

\begin{proof}[Outline of proof of Theorem \ref{brinfty}]
The arguments outlined here are the same as in \cite[4.2, 4.4]{BR19}, to which we refer for a more comprehensive treatment. By \cite[Theorem 7.4 and Remark 7.3]{BW15}, the completed local ring of $\Sigma^i_k(E)$ at $x$ represents the deformation functor
$$\Def^i_k(\Omega_{\dR}(\End L_x),\Omega_{\dR}(E\otimes L_x)).$$
For the definition of the above functor see \cite[Definition 3.15]{BW15}. The quasi-isomorphism between $(\Omega_{\dR}(\End L_x),\Omega_{\dR}(E\otimes L_x))$ and $(A,M)$ induces, by \cite[Theorem 3.16]{BW15}, an isomorphism of functors $\Def^i_k(\Omega_{\dR}(\End L_x),\Omega_{\dR}(E\otimes L_x))\cong \Def^i_k(A,M)$. By the theory of $L_\infty$ pairs, there exists a structure (unique up to isomorphism) of $L_\infty$ pair on $(H^\bullet A, H^\bullet M)$, compatible with induced grading $W$ and such that $(H^\bullet A, H^\bullet M)$ is quasi-isomorphic to $(A,M)$ as $L_\infty$-pairs, see \cite[Theorem 1.6, Corollary 2.20]{BR19}. The quasi-isomorphism of pairs induces then an isomorphism of deformation functors 
$$\Def^i_k(A,M)\cong \Def^i_k(H^\bullet A, H^\bullet M),$$
where now the right-hand side denotes the deformations in the $L_\infty$ sense \cite[Section 3]{BR19}. The advantage of the last isomorphism is that $\Def^i_k(H^\bullet A, H^\bullet M)$ is pro-represented by an explicit ring, namely by the completed local ring at $0$ of the resonance variety
$$\cR^i_k(H^\bullet A, H^\bullet M):= \{ \omega\in H^1 A : \dim_\C H^i(H^\bullet M, d_\omega)\ge k\}.$$
Here $d_\omega := \sum_{n\ge 0} \frac{1}{n!} m_{n+1}(\omega^{\otimes n},\cdot) $ and $m_{n+1}: (H^\bullet A)^{\otimes n}\otimes H^\bullet M\to H^\bullet M$ are the multiplication maps of $H^\bullet M$ as an $L_\infty$ module over $H^\bullet A$, for $n\ge 0$.
As we prove below, the differential $d_\omega$ involved in the definition is a finite sum of terms so the expression above is well-defined.
The fact that the germ at $0$ pro-represents the deformation functor is exactly the argument in \cite[4.2. Proof of Theorem 1.5]{BR19} (the main point is that since $\Omega_{\dR}(\End L_x)=\Omega_{\dR}(\C)$ is a commutative dga, the induced $L_\infty$-structure on $H^\bullet A$ is trivial. In turn, this implies that the Maurer-Cartan equation is trivial and moreover there is no homotopy equivalence to mod-out by).

We now show that there exists a positive integer $N_0$ such that $m_{n+1}(\omega^{\otimes n},\nu)=0$ for all $\omega\in H^1 A, \nu\in H^\bullet M$ and $n\ge N_0$. Note that since $W_0H^1A=0$, every $\omega\in H^1A$ is a sum of elements of positive weight. From this and the fact that the grading $W$ on the pair $(H^\bullet A, H^\bullet M)$ is compatible with the $L_\infty$-module operations (i.e. the weight is additive under multiplication), we deduce that for all $k\in\Z$ and $\nu\in H^k M$ we have that $m_{n+1}(\omega, \ldots, \omega,\nu)\in H^{k+1}M$ is a sum of elements with graded weight $\ge n$. The vector spaces $H^k M$ are finite dimensional and are zero for $k > 2\dim(X)$: in particular, there exists an integer $N_0$ such that for all $h\ge N_0$ and $k\in \Z$, $W_h H^k M = 0$, proving our claim.

We deduce that the sum $d_\omega$ in the definition of the resonance variety is actually a finite sum of terms, whose number is bounded uniformly in $\omega\in H^1 A$. Hence $\cR^i_k(H^\bullet A, H^\bullet M)_{(0)}$ is a quotient of $H^1 A_{(0)}$ by the ideal generated by a finite number of polynomials. Moreover the following square commutes:
$$
\begin{tikzcd}
{\mathcal{R}^i_k(H^\bullet A, H^\bullet M)_{(0)}} \arrow[r, hook] \arrow[d, "\exp"] & H^1 A_{(0)} \arrow[d, "\exp"] \\
\Sigma^i_k(E)_x \arrow[r, hook]                                             & {\bM_{\B}(X,1)_x}       
\end{tikzcd}
$$
The theorem now follows from a bialgebraicity statement for algebraic tori, see \cite[Corollary 2.2]{BW17}: in words, if the germ at $0\in \C^n$ of an algebraic variety is transformed, under the exponential map $\C^n\to (\C^\times)^n$, to the germ at $1$ of an algebraic variety, then the latter is the germ of a subtorus.
\end{proof}


\section{Motivicity of absolute sets}

\subsection{Motivicity in rank one} Let $X$ be a complex algebraic variety. Recall that $\bM_{\B}(X,1)$ is a finite union of algebraic tori defined over $\Q$ and the algebraic tangent space at the identity is canonically identified with $H^1(X,\Q)$. By work of Deligne, it is endowed with a functorial mixed Hodge structure. We recall the definition of motivicity from \cite{EK}. 

\begin{defi}\label{defimotivic}
A $\Q$-subtorus $M\subset \bM_{\B}(X,1)$ is {\em motivic} if $T_eM\subset H^1(X,\Q)$ is a sub-MHS, where $T_eM$ is the algebraic tangent space of $M$ at the identity. 
\end{defi}

In \emph{loc. cit.} it is proved, using Faltings' theorem on the Mordell conjecture, that the cohomology jump loci $\Sigma^i_k(\cF)$ in $\bM_{\B}(X,1)$ are finite unions of translated motivic subtori, where $X$ is an algebraic variety satisfying $W_0H^1(X,\Q) = 0$ and $\cF$ is arithmetic. We prove this fact for certain components of absolute $\C$-constructible subsets for normal varieties. Denote by $\bM_{\B}(X,1)^0$ the connected component of the Betti moduli space containing the constant local system.

\begin{prop}\label{thm_motiv_rank_one}
Let $X$ be a normal quasi-projective variety and $S$ an absolute $\C$-constructible subset contained in $\bM_{\B}(X,1)^0(\C)$. Then the Zariski-closure of an irreducible component of $S$ is a translated motivic subtorus.
\end{prop}

\begin{proof}
Consider first the case where $X$ is smooth. 
Let $\overline{X}$ be a smooth projective compactification of $X$ with normal crossing divisor $D$. 
We can and will assume that $X,\overline{X}, D$ are defined over a common countable subfield $K\subset \C$.
Let $Z$ be an irreducible component of $S$.
By \cite[Proposition 10.4.3]{BW20}, irreducible components of absolute $\C$-constructible subsets in rank one are absolute $\C$-constructible.
By Remark \ref{rem_abs_vs_moduli_abs}, $Z$ is also moduli-absolute $\C$-constructible. 
Let $RH^\an:\bM_{\dR}(\overline{X}/D,1)^\an\to \bM_{\B}(X,1)^\an$ be the analytic Riemann-Hilbert correspondence and $W$ be an analytic irreducible component of $(RH^{\an})^{-1}(Z)$.
Since $Z$ is moduli-absolute $\C$-constructible, we have that for all $\sigma\in\Gal(\C/K)$ the set $\sigma(W)\subset \bM_{\dR}(\overline{X}/D,1)^\an$ is analytically constructible (here we use that $X,\overline{X}, D$ are defined over $K$, hence $\bM_{\dR}(\overline{X}/D,1)$ is as well).
Up to enlarging the field $K$ and preserving its countability, we can also assume that the irreducible component of $\bM_{\dR}(\overline{X}/D,1)$ containing $W$ is defined over $K$.
Then \cite[Theorem 7.4.6]{BW20} implies that the Euclidean closure of $W$ is algebraic (and irreducible) in $\bM_{\dR}(\overline{X}/D,1)$.
Let $\bM$ be the connected component of $\bM_{\dR}(\overline{X}/D,1)$ containing $W$ and $\bM_0$ be the connected component containing the trivial connection.
After choosing a point of $(RH^\an)^{-1}(1)\cap \bM$ as origin, one constructs an algebraic isomorphism $\bM\xrightarrow{\sim} \bM_0$, commuting with $RH^\an$. 
Let $W'$ be the image of $W$ in $\bM_0$.
With these hypotheses in place, the Zariski closure $\overline{Z}$ is an irreducible Betti-de Rham subset of $\bM_{\B}(X,1)^0(\C)$ \cite[Definition 3.1 and Section 4]{BW17} i.e. $\dim(\overline{W'})=\dim(\overline{Z})$ and $RH^\an(\overline{W'})\subset \overline{Z}$.
Now [{\em loc. cit.}, Theorem 3.3] implies that $\overline{Z}$ is a translate of a motivic subtorus. 
We note that the case where $X$ is smooth projective was already treated in \cite{S93}.

Assume now that $X$ is only normal. Let $\widetilde{X}\to X$ be a resolution of singularities, $i: \bM_{\B}(X,1)\hookrightarrow \bM_{\B}(\widetilde{X},1)$ the embedding and $S$ an absolute $\C$-constructible subset contained in $\bM_{\B}(X,1)^0(\C)$. By definition, $i(S)$ is absolute $\C$-constructible in $\bM_{\B}(\widetilde{X},1)(\C)$ and it is contained in $\bM_{\B}(\widetilde{X},1)^0(\C)$. By the smooth case treated above, the closure of an irreducible component is a translated motivic subtorus in $\bM_{\B}(\widetilde{X},1)$. To conclude that the closure of an irreducible component of $S$ is motivic in $\bM_{\B}(X,1)$, it suffices to observe that $i(\bM_{\B}(X,1)^0)$ is itself a motivic subtorus of $\bM_{\B}(\widetilde{X},1)$ since the pullback $H^1(X,\Q)\to H^1(\widetilde{X},\Q)$ is a morphism of Hodge structures.
\end{proof}

By applying the previous proposition to the cohomology jump loci we obtain the motivicity of cohomology jump loci mentioned in the introduction:
\begin{cor}\label{cor_motiv_rank_one}
Let $X$ be a normal quasi-projective variety and $\cF\in D^b_c(X,\C)$. The subtori corresponding to the irreducible components of $\Sigma^i_k(\cF)\cap \bM_{\B}(X,1)^0(\C)$ are motivic subtori.
\end{cor}
\begin{proof}
If $S\subset \bM_{\B}(X,1)(\C)$ is an absolute $\C$-closed subset then the intersection $S\cap \bM_{\B}(X,1)^0(\C)$ is absolute $\C$-closed as well (to see this, embed $\bM_{\B}(X,1)$ into $\bM_{\B}(\widetilde{X},1)$ and apply the fact that taking irreducible or connected components preserves absolute constructibility \cite[Proposition 10.4.3]{BW20}). Now the result follows from Proposition \ref{thm_motiv_rank_one} and the fact that $\Sigma^i_k(\cF)$ is absolute $\C$-closed.
\end{proof}

\subsection{Motivicity in higher rank}\label{higherrank} For the purpose of studying motivicity in higher rank, it is more convenient to work with the representation variety $\bR(X,x_0,n)$ instead of $\bM_{\B}(X,n)$ (note that they coincide in rank one). For $n\ge 2$, the space $\bR(X,x_0,n)$ is no longer an algebraic group and the global geometry of ``special subvarieties" such as the cohomology jump loci is not well understood. Nonetheless there is a good description of their formal geometry at a point. This is the essential ingredient in what follows.

We recall here some basic definitions of the objects we will deal with. A more detailed treatment can be found in \cite[\S 1]{ES}. 

\begin{defi}\label{hodge}
Let $w\in\Z$ be an integer. A \emph{$\C$-Hodge structure of weight $w$}  is the datum of a $\C$-vector space $V$ (allowed to be infinite dimensional) endowed with two decreasing filtrations $F^\bullet, \overline{G}^\bullet$ such that $Gr_F^pGr_{\overline{G}}^q V = 0$ unless $p+q = w$. A \emph{polarization} $S$ on a finite dimensional $\C$-Hodge structure $V$ of weight $w$ is an Hermitian form such that the Hodge decomposition of $V$ is orthogonal with respect to $S$ and $(-1)^{p+w}S|_{H^{p,q}}$ is positive definite.

A \emph{$\C$-mixed Hodge structure} is the datum of $\C$-vector space $V$ endowed with two decreasing filtrations $F^\bullet, \overline{G}^\bullet$ and an increasing filtration $W_\bullet$ such that $Gr_W^w V$, endowed with the induced filtrations $F^\bullet$, $\overline{G}^\bullet$ is a $\C$-Hodge structure of weight $w$. A $\C$MHS $V$ is {\em defined over} $\R$ if $V$ is the complexification of a real vector space, the filtration $W_\bullet$ is defined over $\R$ and $\overline{G}^\bullet$ is the complex conjugate of $F^\bullet$. 

Let $X$ be a complex manifold. A \emph{$\C$-variation of Hodge structures of weight $w$} is the datum of a $C^\infty$ complex vector bundle $\cV$ on $X$, a decomposition $\cV = \oplus_{p+q = w} \cV^{p,q}$ of $C^\infty$ bundles and a flat connection $\nabla: \cV\to \cE^1\otimes \cV$ satisfying Griffiths' transversality:
$$\nabla: \cV^{p,q}\to \cE^{1,0}(\cV^{p-1,q+1})\oplus\cE^{1,0}(\cV^{p,q})\oplus\cE^{0,1}(\cV^{q,p})\oplus\cE^{0,1}(\cV^{q+1,p-1}),$$
where $\cE^n$ (resp. $\cE^{a,b}$) denotes the sheaf of smooth $n$-forms (resp. smooth forms of type $(a,b)$). A {\em polarization} on a $\C$VHS $(\cV,\nabla)$ of weight $w$ is a flat Hermitian form $S$ such that $\cV_x = \oplus \cV_x^{p,q}$ is a $\C$-Hodge structure of weight $w$ polarized by $S$, for all $x\in X$.
\end{defi}

Let $X$ be a smooth projective variety, $x_0\in X(\C)$ a base point and $\rho \in \bR(X,x_0,n)(\C)$ a representation underlying a polarized $\C$VHS. Let $\cO_\rho$ be the local ring at $\rho$ and let $\fm$ be the maximal ideal of its completion $\widehat{\cO}_\rho$. The ring $\widehat{\cO}_\rho$ is endowed with a functorial $\C$MHS and if the representation $\rho$ is defined over $\R$, the MHS is defined over $\R$ and is split, i.e. it is isomorphic as $\R$MHS to a direct sum of pure Hodge structures defined over $\R$. Its construction is due implicitly to Goldman and Millson in \cite{GM} and spelled out in \cite[Section 2.3]{ES}. We recall its definition here, following \emph{loc. cit}. 

Denote by $L_\rho$ the local system on $X$ defined by $\rho$ and $\epsilon_{x_0}: H^0(X,\End L_\rho)\to (\End L_\rho)_{x_0}$ the evaluation at $x_0$. The {\em quadratic cone} at $\rho$ is the algebraic variety given by
$$\cQ(L_\rho) := \{ \omega \in H^1(X,\End L_\rho): [\omega,\omega] = 0\}\subset H^1(X,\End L_\rho).$$
The main result in \cite{GM} constructs an isomorphism between $\widehat{\cO}_\rho$ and the completion of the local ring at the origin in the space \begin{equation}\label{split}
\cQ(L_\rho) \times (\End L_\rho)_{x_0}/ \im\,\, \epsilon_{x_0}.
\end{equation}
The completed local ring at the origin in $\cQ(L_\rho)$ is given by 
\begin{equation}\label{quot}
Sym^\bullet H^1(X,\End L_\rho)^\vee/ I
\end{equation}
where $I$ is the homogeneous ideal generated by $I_2 := \im(H^2(X,\End L_\rho)^\vee \to Sym^2 H^1(X,\End L_\rho)^\vee)$ and the latter map is the transpose of the cup product
$$[-,-] : H^1(X,\End L_\rho)\otimes H^1(X,\End L_\rho)\to H^2(X,\End L_\rho).$$
For $n\ge 3$ define the subspaces
$$I_n:= I_2 \otimes Sym^{n-2}H^1(X,\End L_\rho)^\vee \subset Sym^nH^1(X,\End L_\rho)^\vee .$$
If we define $I_0 = I_1 = 0$,  then $I_n$ corresponds to the degree $n$ piece of $I$. Thus $\widehat{\cQ(L_\rho)}_{(0)}$ is given by
\begin{equation}\label{iso}
\prod_{n\ge 0} Sym^nH^1(X,\End L_\rho)^\vee /I_n.
\end{equation}
By the generalized K\"ahler identities for polarized $\C$VHS (see \cite[Section 2]{Z79} for real VHS), the cohomology groups $H^n(X,\End L_\rho)$ carry a functorial $\C$MHS of weight $n$ and the cup product is a morphism of MHS. The direct product in (\ref{iso}) then carries a (infinite-dimensional) $\C$MHS, and the direct factor in place $n$ is pure of weight $-n$. A similar discussion applies to the second factor in (\ref{split}): the completed local ring at $0$ of $(\End L_\rho)_{x_0}/ \im\,\, \epsilon_{x_0}$ is naturally an infinite dimensional $\C$-Hodge structure of weight zero.

\begin{defi}
Let $\rho \in \bR(X,x_0,n)(\C)$ be a representation underlying a polarized $\C$VHS. The {\em canonical split} $\C$MHS on $\widehat{\cO}_\rho$ is the $\C$MHS constructed in the preceding paragraph.
\end{defi}

Define the cohomology jump loci in $\bR(X,x_0,n)$ with indices $i,k\ge 0$ in the obvious way:
$$\widetilde{\Sigma}^i_k := \{ \rho\in \bR(X,x_0,n)(\C) : \dim_\C H^i(X,L_\rho)\ge k\}.$$

\begin{lem}
Let $X$ be a smooth projective variety. For every $i,k\in\Z_{\ge 0}$, the cohomology jump locus $\widetilde{\Sigma}^i_k$ contains representations underlying polarized $\C$VHS. 
\end{lem}

\begin{proof}
It is sufficient to prove that the image of $\widetilde{\Sigma}^i_k$ in $\bM_{\B}(X,n)$, which is $\Sigma^i_k$, contains $\C$VHS. We claim that ${\Sigma}^i_k$ is invariant under the action of $\C^\times$. This follows from \cite[Lemma 2.2]{S92} where it is proved that if $E$ is a harmonic bundle then there are natural isomorphisms between the Dolbeault cohomology of $E$ and the de Rham cohomology of $E$. Here the Dolbeault cohomology of $E$ is the cohomology of the complex
$$ 0\to \cE\xrightarrow{\theta}\cE\otimes_{\cO_X} \Omega^1\xrightarrow{\theta}\cE\otimes_{\cO_X} \Omega^2\to \ldots .$$
where $(\cE,\theta)$ is the Higgs bundle associated to $E$. Hence the jump locus on the Betti side corresponds to the jump locus on the Dolbeault side. Moreover, if $(\cE,\theta)$ is a Higgs bundle and $\lambda\in\C^\times$ then $\dim_\C H^i_{\Dol}(\cE,\lambda\theta) = \dim_\C H^i_{\Dol}(\cE,\theta)$, which proves that ${\Sigma}^i_k$ is $\C^\times$-invariant. To conclude observe that ${\Sigma}^i_k$ is closed. This in turn implies that $\Sigma^i_k$ contains $\C^\times$-invariant points, i.e. polarized $\C$VHS. Indeed, \cite[Proposition 1.4]{S92} implies that if $(\cE,\theta)$ is a Higgs bundle then $\lim_{t\to 0}(\cE,t\theta)$ exists in the Dolbeault moduli space and is $\C^\times$-invariant.
\end{proof}
\begin{thm}\label{tangent}
Let $X$ be a smooth projective variety with a base point $x_0$, $\rho\in\bR(X,x_0,n)(\C)$ a representation underlying a polarized $\C$VHS. Let $i\ge 0$ and $k=\dim H^i(X,L_\rho)$. Then the ideal cutting $\tilde{\Sigma}^i_k$ in $\widehat{\cO}_\rho$ is a sub-MHS of the canonical split $\C$MHS.
\end{thm}

\begin{proof}
By \cite[\S 7]{BW15} the infinitesimal structure of $\widetilde{\Sigma}^i_k$ around $\rho$ is governed by the deformation theory of the augmented dgla-pair $(\cA^\bullet(\End L_\rho), \cA^\bullet(L_\rho), \epsilon_{x_0})$, where $\cA^\bullet(-)$ denotes the twisted de Rham complex of $C^\infty$ forms and the augmentation $\epsilon_{x_0}: \cA^0(\End L_\rho)\to (\End L_\rho)_{x_0}$ is the evaluation at $x_0$. Let $R_\rho$ be the local ring of $\widetilde{\Sigma}^i_k$ at $\rho$. The resonance variety is defined by
$$\cR^i_k(L_\rho) := \{ \omega\in H^1(X,\End L_\rho): [\omega, \omega] = 0 \,\,\text{and}\,\, \dim_\C H^i(H^\bullet(X,L_\rho), \omega\wedge) \ge k\},$$
where $(H^\bullet(L_\rho), \omega\wedge)$ is the \emph{Aomoto complex}:
$$\ldots \rightarrow H^r(X,L_\rho)\xrightarrow{\omega\wedge} H^{r+1}(X,L_\rho)\rightarrow\ldots\,\,.$$

Since $\rho$ underlies a $\C$VHS it is semisimple and \cite[7.8]{BW15} implies that the inclusion $\widehat{R}_\rho\subset \widehat{\cO}_\rho$ followed by the Goldman-Millson isomorphism, identifies $\widehat{R}_\rho$ with the completed local ring at $0$ of 
$$\cR^i_k(L_\rho)\times (\End L_\rho)_{x_0}/ \im\,\, \epsilon_{x_0}.$$
To prove the theorem it is enough to prove that the ideal cutting $\widehat{\cR^i_k(L_\rho)_{(0)}}$ in $\widehat{\cQ(L_\rho)_{(0)}}$ is a sub-MHS of the split MHS described above.

Let $Z$ be the subset of $\omega\in H^1(X,\End L_\rho)$ such that the linear map
$$\omega\wedge(-)\oplus\omega\wedge(-): H^{i-1}(X,L_\rho)\oplus H^i(X,L_\rho)\to H^{i}(X,L_\rho)\oplus H^{i+1}(X,L_\rho)$$
has rank bigger than $\dim H^i(L_\rho)-k+1 = 1$. It has a natural scheme structure given by the zero locus of a certain determinantal ideal. By \cite[Section 2]{BW15}, whenever $\omega$ is a Maurer-Cartan element, then $\omega\in Z$ if and only if $\dim_\C H^i(H^\bullet(X,L_\rho), \omega\wedge) \ge k$. The scheme structure on the cohomology jump locus is then given by the scheme theoretic intersection of the quadratic cone with $Z$ (see \cite[Definition 4.4]{BW15}). Let $E$ be the vector space $H^1(X,\End L_\rho)$ and $R = Sym^\bullet E^\vee$. We show that the ideal $J\subset \widehat{R}$ that cuts (the completion of) $Z_{(0)}$ is a sub-MHS, concluding the proof.

Let $A_n:= \widehat{R}/\fm^n$. The ring $A_n$ is isomorphic as $\C$MHS to $\bigoplus_{r = 0}^{n-1} \Sym^r E^\vee$ and the algebra operation $A_n\otimes A_n\to A_n$ is a morphism of $\C$MHS. Let
\begin{equation}\label{map}
\phi_n: (H^{i-1}\oplus H^i)\otimes A_n \xrightarrow{\omega_t\wedge\,\, \oplus\, \,\omega_t \wedge} (H^i\oplus H^{i+1})\otimes A_n
\end{equation}
be the wedge with the tautological element $\omega_t\in E\otimes E^\vee\subset E\otimes A_n$. We show that the map $\phi_n$ is a morphism of $\C$MHS. Let $V_1 = H^{i-1}\oplus H^i$ and $V_2 = H^i\oplus H^{i+1}$. The cup-product corresponds to a morphism $\phi: E\otimes V_1\to V_2$ of $\C$MHS. Tensoring $\phi$ with $A_n$ we get a morphism
$$(E\otimes A_n)\otimes V_1 \to V_2\otimes A_n.$$
Fixing the tautological element (which is of type $(0,0)$) in the summand $E\otimes E^\vee\subset E\otimes A_n$, we get a morphism of $\C$MHS
\begin{equation}\label{eq}
V_1\to V_2\otimes A_n
\end{equation}
and $\phi_n$ is obtained by extending $A_n$-linearly:
$$V_1\otimes A_n \to V_2\otimes A_n \otimes A_n \to V_2\otimes A_n,$$
where the first map is (\ref{eq}) tensored with $A_n$ and the second map is obtained by tensoring the identity on $V_2$ with the multiplication map $A_n \otimes A_n \to A_n$. This shows that $\phi_n$ is a morphism of $\C$MHS. In down to earth terms, this means that the map between the free $A_n$-modules in (\ref{map}) can be given by a matrix with elements in $E^\vee$ that are pure of some bidegree. By \cite[Corollary 2.5 and \S 7]{BW15} the ideal $J_n := J\otimes \widehat{R}/\fm^n$, corresponding to the $n$-th infinitesimal neighborhood of $0$ in $Z$, is the ideal generated by the coefficients of this matrix. Since the coefficients are pure with respect to the bigrading, $J_n$ is compatible with the bigrading.
\end{proof}

\begin{rem}
As explained in \cite[Introduction]{ES}, the split $\C$MHS on $\widehat{\cO}_\rho$ is only an approximation of the ``true" $\C$MHS, which is constructed in {\em loc. cit.} and in \cite{L1}. It is clear that the natural framework of the above theorem should be such a $\C$MHS. After a first version of this paper appeared, Louis-Cl\'ement Lef\`evre proved in \cite[Theorem 1.1]{L3} that Theorem \ref{tangent} holds for this $\C$MHS and for arbitrary indices $i,k\ge 0$.
\end{rem}

We record here that a similar calculation is carried out in \cite[Theorem 1.7]{BR19}, where the authors give a description of the tangent space to the cohomology jump locus associated to an $L_\infty$-pair. Namely, if $(C,M)$ is a $L_\infty$ pair and $h_i=\dim H^iM$ then
$$T\text{Def}^i_k(C,M) = 
\begin{cases}
H^1C & \text{if }\,\,k < h_i,\\
\ker\left( H^1C\to \oplus_{j=i-1,i} \Hom(H^jM, H^{j+1}M) \right) & \text{if }\,\,k = h_i,\\
0 & \text{if }\,\,k > h_i,
\end{cases}
$$
where, in the second case, the map is given by the module structure of $H^\bullet M$ over $H^\bullet C$. It is easy to see that if $C$ and $M$ carry a weight filtration and Hodge filtrations compatible with the $L_\infty$ operations then the tangent space to the cohomology jump loci inherits a $\C$MHS.


\section{Jump loci of Hodge theoretic data}

In this section we complement some results of \cite{B09} about the structure of jump loci associated to Hodge theoretic numerical data. 

Let $X$ be a complex algebraic variety. If $L$ is a unitary local system, then Saito's theory of mixed Hodge modules endows the cohomology groups $H^i(X,L)$ with functorial $\C$MHS's. If we fix $k,p,q\in\Z_{\ge 0}$, one can ask about the structure of the loci over which one of the quantities
$$\dim_\C Gr^F_pH^k(X,L),\,\,\dim_\C Gr_W^{p+q}H^k(X,L), \,\,\dim_\C Gr^F_pGr_W^{p+q}H^k(X,L)$$
remains constant, as $L$ varies in the real subvariety $\bU(X,n)\subset \bM_{\B}(X,n)$ of unitary local systems. The same question can be asked with ordinary cohomology replaced by intersection cohomology, or cohomology with compact supports. The most general framework for the problem would involve arbitrary operations in the category of mixed Hodge modules.

We prove the following result for rank one local systems on smooth varieties:

\begin{thm}\label{hodgejump}
Let $U$ be a smooth variety and let $\cA$ be the Boolean algebra generated by torsion-translated real subtori of $\bU(U,1)$.

(1) For all $a,i,k\in\Z_{\ge 0}$, the set
$$ W^{a,i,k}:= \{L \in \bU(U,1):   \dim_\C Gr_W^aH^i(U,L)\ge k\} $$
belongs to $\cA$.

(2) Let $j: U\to X$ be an open embedding into a smooth proper variety $X$ such that $X\setminus U$ is a simple normal crossing divisor. For all $p,q,k\in\Z_{\ge 0}$, the set 
$$V^{p,q,k} := \{L \in \bU(U,1):  \dim_\C Gr^F_p \IH^{p+q}(X,L)\ge k\}$$
belongs to $\cA$. Here $\IH$ denotes intersection cohomology.
\end{thm}

\begin{proof}
(1) Let $X$ be a smooth compactification of $U$ by a simple normal crossing divisor $D\subset X$. Up to a suitable renumbering of the indices the weight filtration on $H^i(X,L)$ coincides with the Leray filtration associated to the embedding $j:U\hookrightarrow X$ (see Remark to Theorem 7.1 in \cite{T87}). By \cite[Corollary 6.4.4]{BW20}, the jumping loci for the dimension of the graded pieces for Leray filtration induced on $H^i(X,L)$ are absolute $\Q$-constructible, hence belong to the Boolean algebra generated by torsion-translated subtori. Intersecting with $\bU(U,1)$ we obtain the result.

\medskip

(2) By Theorem \ref{mainres}, for any $i,k\ge 0$, the set
$$S^i_k:=\{L \in \bM_{\B}(U,1)(\C):  \dim_\C \IH^i(X,L)= k\}$$
belongs to the Boolean algebra generated by torsion-translated subtori. By the Hodge theory of unitary local systems \cite{T87}, the spectral sequence converging to $\IH^\bullet(X,L)$ associated to the Hodge filtration degenerates at the first page, whenever $L$ is unitary. We will first partition $\bM_{\B}(U,1)$ into subvarieties belonging to the Boolean algebra of torsion-translated subtori and prove that the dimension of $\prescript{}{F}E_1^{p,q}$ varies in an upper semi-continuous way as $L$ varies in $V\cap\bU(U,1)$, for all strata $V$ of the partition. The argument used to conclude is then the same as in \cite{S93}: one writes $\dim_\C \IH^i = \sum_{a} \dim\prescript{}{F}E_1^{a, i-a}$ and when restricted to a subset of $V\cap\bU(U,1)$ where $\dim_\C \IH^i$ is constant, the terms of the sum have to stay constant, by upper-semicontinuity, therefore proving the claim on $V\cap\bU(U,1)$, which in turn proves the theorem. To identify the $E_1$-page we will use the results in \cite{T87}. Note also that in the present case of the complement of a normal crossing divisor in a smooth variety, we have $\IH^\bullet(X,L)\cong H^\bullet(X,j_*L)$ (see e.g. [{\em loc. cit.}, Proposition 5.4]).

Let $D:=X\setminus U$ and write $D = \cup_{i\in I} D_i$ where $D_i$ are the irreducible components. For all $i\in I$, let $D^*_i:=D_i\setminus\cup_{j\neq i} (D_i\cap D_j)$ and, for all subsets $J\subset I$, define $D_J:=\cup_{j\in J} D_j$ and $U_J:=U\cup \bigcup_{j\in J} D^*_j$. 
For each local system $L\in\bM_{\B}(U,1)(\C)$ there corresponds a subset $J\subset I$ such that $L$ has trivial monodromy around $D_j, j\in J$ and non-trivial monodromy around $D_i, i\notin J$. We can partition $\bM_{\B}(U,1)$ as 
$$\bM_{\B}(U,1) = \bigsqcup_{J\subset I} V_J,$$
where $V_J$ is the set of local systems with trivial monodromy around $D_j$, $j\in J$ and nontrivial monodromy around $D_i, i\notin J$. Clearly $V_J$ is an open (possibly empty) subset of $\bM_{\B}(U_J,1)(\C)$ and it belongs to the Boolean algebra generated by torsion-translated subtori of $\bM_{\B}(U,1)$. 
Fix $L\in V_J\cap \bU(U,1)$. By \cite[Theorem 5.1.(a)]{T87}, we have $\prescript{}{F}E_1^{p,q} = H^q(X, \Omega^p_X(\log D_{I\setminus J})\otimes \cL^{>-1})$, where $\cL^{>-1}$ is the Deligne extension of $L$. (Note that in {\em loc. cit.}, certain sheaves denoted by $\widetilde{\Omega}^p_X(\cL^{>-1})$ are introduced, but if the local system has no trivial monodromies around the boundary divisor then they reduce to the usual twisted sheaves of log-forms.) 
In $\bM_{\dR}(U_J,1)$, choose the logarithmic connection $(\cL^{>-1},\nabla)$ corresponding to the Deligne extension of $L$. There exists a neighborhood $\cB$ of $L$ in $V_J^\an$ and a neighborhood $\cB'$ of $(\cL^{>-1},\nabla)$ in $\bM_{\dR}(U_J,1)^\an$ with $RH^\an(\cB')=\cB$ and such that, if $(E,\nabla)\in \cB'$ then $(E,\nabla)$ is the Deligne extension of $RH^\an(E,\nabla)$. Indeed, as a logarithmic connection $(E,\nabla)$ varies in $\bM_{\dR}(U_J,1)$, its vector of residues around the divisor $D_{I\setminus J}$ varies continuously (in fact, algebraically). 
As $(E,\nabla)$ varies in $\cB'$, $\dim H^q(X, \Omega^p_X(\log D_{I\setminus J})\otimes E)$ varies upper-semicontinuously. 
Therefore, as $L$ varies in $\cB\cap \bU(U,1)$ the terms $\dim\prescript{}{F}E_1^{p,q}$ vary upper-semicontinuously.
As explained in the first paragraph, this implies that as $L$ varies in $\cB\cap \bU(U,1)\cap S^{p+q}_k$, for $k\in \Z_{\ge 0}$, the terms $\prescript{}{F}E_1^{p,q}$ have constant dimension. 
The partition of $V_J \cap \bU(U,1)$ defined by the dimensions of $Gr_F^p \IH^{p+q}$ consists of locally closed real-analytic subvarieties and the above argument proves that around each point there is a small neighborhood such that said partition coincides with the one induced by the dimensions of $\IH^{p+q}$. The latter partition is contained in the Boolean algebra generated by torsion-translated real subtori, by Theorem \ref{mainres}.
\end{proof}

In the notation of the previous theorem, the same arguments prove that the stratification of $\bU(U,1)$ given by $\dim_\C\Gr^F_p H^{p+q}(U,L)$ belongs to $\cA$. This implies for example \cite[Theorem 1.3]{B09} and allows one to generalize the assumptions on $\cW$ in Theorem 1.4 of {\em loc. cit.}: namely, $\cW$ can be any unitary local system on $U$ coming from geometry.

\begin{rem}
Let $j: U\hookrightarrow X$ as in Theorem \ref{hodgejump} (2). It is unclear to us whether the jumping loci of \emph{mixed} Hodge numbers of $H^\bullet(U,L)$ determine torsion-translated real subtori $\bU(U,1)$. Let $h^{p,q,k} = \dim_\C Gr_p^FGr_W^{p+q}H^k(U,L)$. The spectral sequence corresponding to the weight filtration degenerates at $E_2$ and it is endowed with a Hodge filtration $F$. The number $h^{p,q,k}$ is equal to the dimension of $Gr^F_p$ applied to the middle cohomology of the complex
$$_{W}E_1^{-m-1,k+m}\xrightarrow{d_1}~ _{W}E_1^{-m,k+m}\xrightarrow{d_1}~ _{W}E_1^{-m+1,k+m},$$
where $m=p+q$. By strictness of the differentials with respect to $F$, one can invert the operations of taking cohomology and applying $Gr^F$. The terms of the spectral sequence are described in \cite[Proposition 1.7 (b)]{T87} and $h^{p,q,k}$ is the dimension of the middle cohomology of
\begin{align*}
& Gr^F_p H^{k-m-2}(D^{(m+1)}, \Omega_{D^{(m+1)}}^\bullet \otimes \cM_{m+1}) \\
\to\quad & Gr^F_pH^{k-m}(D^{(m)}, \Omega_{D^{(m)}}^\bullet \otimes \cM_m) \\
\to\quad& Gr^F_p H^{k-m+2}(D^{(m-1)}, \Omega_{D^{(m-1)}}^\bullet \otimes \cM_{m-1})
\end{align*}
as $L$ varies: here $D^{(n)}$ denotes the normalization of the $n$-fold intersections of the $D_i$'s and $\cM_{n}$ is the canonical extension of $j_*L|_{D^{(n)}\setminus D^{(n+1)}}$. The graded pieces are in turn identified with the cohomology over $D^{(n)}$, $n=m-1,m,m+1$, of certain {\em coherent} sheaves depending on $L$: their dimension therefore varies in an upper-semicontinuous way but we are unable to control the dimension of the cohomology of the maps between them.
\end{rem}

\end{document}